\documentclass[11pt,a4paper]{article}
\usepackage{amssymb,amsmath,amsfonts,amsthm,array,bm,color}
\usepackage{amsmath}
\usepackage{amsfonts}
\usepackage{mathrsfs}
\usepackage{graphicx}

\newcommand{\R}{\mathbb R}

\newcommand{\Pp}{\mathbb P}
\newcommand{\Ee}{\mathbb E}

\newcommand{\I}{{\bf 1}}

\def\<{\langle}
\def\>{\rangle}
\def\essinf{\mathop{\rm ess\,inf}}

\newcommand{\var}{\textmd{Var}}
\newcommand{\Bb}{\mathscr{B}}

\newcommand{\scalar}[1]{\langle#1\rangle}

\newtheorem{theorem}{Theorem}[section]
\newtheorem{lemma}[theorem]{Lemma}
\newtheorem{proposition}[theorem]{Proposition}
\newtheorem{corollary}[theorem]{Corollary}

\theoremstyle{definition}

\newtheorem{example}[theorem]{Example}
\newtheorem{remark}[theorem]{Remark}

\begin{document}

\allowdisplaybreaks

\title{\bfseries Refined basic couplings and Wasserstein-type distances for SDEs with L\'{e}vy noises}

\author{Dejun Luo\footnote{Email: luodj@amss.ac.cn. RCSDS, Academy of Mathematics and Systems Science, Chinese Academy of Sciences, Beijing 100190, China, and School of Mathematical Sciences, University of the Chinese Academy of Sciences, Beijing 100049, China. } \ and Jian Wang\footnote{Email: jianwang@fjnu.edu.cn. College of Mathematics and Informatics \& Fujian Key Laboratory of Mathematical Analysis and Applications (FJKLMAA), Fujian Normal University, 350007, Fuzhou, P.R. China.} }

\date{}

\maketitle

\makeatletter

\renewcommand\theequation{\thesection.\arabic{equation}}
\@addtoreset{equation}{section} \makeatother

\begin{abstract}
We establish the exponential convergence with respect  to the $L^1$-Wasserstein distance and the total variation for the semigroup corresponding to the stochastic differential equation
  $$d X_t=d Z_t+b(X_t)\,d t,$$
where $(Z_t)_{t\ge0}$ is a pure jump L\'{e}vy process whose L\'{e}vy measure $\nu$ fulfills
  $$ \inf_{x\in \R^d, |x|\le \kappa_0} [\nu\wedge (\delta_x \ast \nu)]( \R^d)>0$$
for some constant $\kappa_0>0$, and the drift term $b$ satisfies that for any $x,y\in \R^d$,
  $$\langle b(x)-b(y),x-y\rangle\le \begin{cases}
  \Phi_1(|x-y|)|x-y|,&  |x-y|< l_0;\\
  -K_2|x-y|^2,&  |x-y|\ge l_0
  \end{cases}$$
with some positive constants $K_2, l_0$ and positive measurable function $\Phi_1$. The method is based on the refined basic coupling for L\'{e}vy jump processes. As a byproduct, we obtain sufficient conditions for the strong ergodicity of the process $(X_t)_{t\ge0}$.
\medskip


\medskip

\end{abstract}

\section{Introduction and Main Results}\label{section1}

In this paper we study the following $d$-dimensional stochastic differential equation (SDE) with jumps
  \begin{equation}\label{s1}
  d X_t=b(X_{t})\,d t+ d Z_t,\quad X_0=x\in \R^d,
  \end{equation}
where $b: \R^d\rightarrow\R^d$ is a measurable function, and $Z=(Z_t)_{t\ge0}$ is a pure jump L\'{e}vy process on $\R^d$.

Throughout this paper, we suppose that \emph{the SDE \eqref{s1} has a non-explosive and pathwise unique strong solution, and $b$ satisfies the assumption ${\mathbf B(\Phi_1(r),\Phi_2(r), l_0)}$ that for any $x,y\in\R^d$,
  \begin{equation}\label{th111}
  \frac{\langle b(x)-b(y), x-y\rangle}{|x-y|}\le \Phi_1(|x-y|)-\big[\Phi_1(|x-y|)+\Phi_2(|x-y|)\big]\I_{\{|x-y|\ge l_0\}},
  \end{equation}
where $\Phi_1$ and $\Phi_2$ are two nonnegative measurable functions, and $l_0\ge0$ is a constant}. For example, when $\Phi_2(r)=K_2r$ for some positive constant $K_2$, ${\mathbf B(\Phi_1(r),\Phi_2(r), l_0)}$ is reduced into ${\mathbf B(\Phi_1(r),K_2r, l_0)}$:
  \begin{equation}\label{th11}
  \begin{split}
  \langle b(x)-b(y),x-y\rangle \leq
  \begin{cases}
 \Phi_1(|x-y|)|x-y|, & |x-y|< l_0;\\
  -K_2|x-y|^2, & |x-y|\ge l_0.
  \end{cases}
  \end{split}
  \end{equation}
This holds if the drift term $b$ is dissipative outside some compact set. In particular, when $\Phi_1(r)=K_1 r$ for some constant $K_1\ge0$,  it follows from  \eqref{th11}  that  for any $x\in\R^d$,
  $$ \langle b(x), x\rangle \le \langle b(0), x\rangle+K_1|x|^2\le C_1(1+|x|^2),$$
which, along with \eqref{th11}, yields that the SDE \eqref{s1} has a non-explosive and  pathwise unique strong solution, see \cite[Chapter 6, Theorem 6.2.3]{AP} (in the standard Lipschitz case) or \cite[Theorem 2]{GK} and \cite[Chapter 3, Theorem 115]{RU} (in the one-sided  Lipschitz case). Note that, since we are sometimes concerned with only measurable drift term $b$, non-Lipschitz condition like ${\mathbf B(K_1,K_2r, l_0)}$ will also be adopted in our results below. The reader can refer to \cite{CSZ, P1, P, TTW, Zhang14} and references therein for recent studies on the existence and uniqueness of strong solution to \eqref{s1} with non-regular drift term. In particular, assuming that $Z$ is the truncated symmetric $\alpha$-stable process on $\R^d$ with $\alpha\in (0,2)$, and $b$ is bounded and $\beta$-H\"older continuous with $\beta> 1-\alpha/2$, it was proved in \cite[Corollary 1.4(i)]{CSZ} that the SDE \eqref{s1} has a unique strong solution for each $x\in\R^d$. Furthermore, in the one-dimensional case, if $\alpha>1$, then the SDE \eqref{s1} also enjoys a unique strong solution for each $x\in\R$, even if  the drift $b$ is only bounded and measurable (see \cite[Remark 1, p.\ 82]{TTW}).

Denote by $\nu$ the L\'evy measure of the pure jump L\'{e}vy process $Z$. We assume that \emph{there is a constant $\kappa_0>0$ such that}
  \begin{equation}\label{th10}
  \inf_{x\in \R^d,\, |x|\le \kappa_0} \big[\nu\wedge (\delta_x\ast \nu)\big]( \R^d)>0.
  \end{equation}
Condition \eqref{th10} was first used in \cite{SW} to study the coupling property of L\'{e}vy processes. It is satisfied by a large class of L\'evy measures. For instance, if
  $$\nu(dz)\ge \I_{B(z_0, \varepsilon)}\rho_0(z)\,dz$$
for some $z_0\in\R^d$ and some $\varepsilon>0$ such that $\rho_0(z)$ is positive and continuous on $B(z_0,\varepsilon)$, then such L\'{e}vy measure $\nu$ fulfills \eqref{th10}, see \cite[Proposition 1.5]{SW2} for details. Actually, as shown in Proposition \ref{ppp}, the condition \eqref{th10} implies that there is a nonnegative measurable function $\rho$ on $\R^d$ such that $\nu(dz)\ge \rho(z)\,dz$ and
  $$\inf_{{x\in\R^d, |x|\le \kappa_0}} \int_{\R^d}[\rho(z)\wedge \rho(z+x)]\,dz>0.$$

Let $(P_t)_{t\ge0}$ be the transition semigroup associated with the process $(X_t)_{t\ge0}$. In this paper we are interested in the asymptotics of the Wasserstein-type distances (including the $L^1$-Wasserstein distance and the total variation) between probability distributions $\delta_x P_t=P_t(x,\cdot)$ and $\delta_y P_t=P_t(y,\cdot)$ for any $x,y\in\R^d$, when the drift term $b$ is dissipative outside some compact set, i.e. $b$ satisfies ${\mathbf B(\Phi_1(r), K_2r, l_0)}$ for some positive measurable function $\Phi_1$, and some constants $K_2>0$ and $l_0\ge0$.

This kind of problems have already been studied by Eberle
\cite{Eberle11, Eberle} in the diffusion case, i.e., the pure jump
L\'{e}vy process $(Z_t)_{t\ge0}$ in \eqref{s1} is replaced by a
Brownian motion $(B_t)_{t\ge0}$. He proved that the
$L^1$-Wasserstein distance between $\delta_x P_t$ and $\delta_y P_t$
decays exponentially fast. This result was slightly strengthened in
\cite{LW}, where we obtained some convergence result with respect to
the $L^p$-Wasserstein distance for any $p\ge1$. In the general
settings of Riemannian manifold and of SDEs with multiplicative
noises, F.-Y. Wang \cite{Wang16} obtained the exponential decay in
the $L^2$-Wasserstein distance under ${\mathbf B(K_1r, K_2 r,
l_0)}$, i.e., \eqref{th11} holds with $\Phi_1(r)=K_1r$ for some
$K_1>0$; moreover, he establishes similar results for the $L^p$-Wasserstein
distance for all $p\ge1$ provided that the diffusion
semigroup is ultracontractive.
 Some developments in the jump case can be found in \cite{Wang15, M15} under ${\mathbf B(K_1r, K_2 r, l_0)}$.
 In particular, the second author \cite{Wang15} obtained exponential convergence rate in
 the $L^p$-Wasserstein distance for any $p\ge1$ when the L\'{e}vy noise in \eqref{s1} has
 an $\alpha$-stable component.
 In the recent paper \cite{M15}, Majka considered a
 larger class of L\'{e}vy processes without $\alpha$-stable components,
 and obtained the exponential convergence rates with respect to both the $L^1$-Wasserstein distance
 and the total variation. See the remarks at the end of Subsection \ref{suce1} for more detailed discussions. We mention that in \cite{M15} the associated L\'{e}vy measure of the L\'{e}vy process $Z$ essentially has a rotationally invariant absolutely continuous component.

In order to present our results, we first introduce some notations. Let $\psi$ be a strictly increasing function on $[0,\infty)$ satisfying $\psi(0)=0$. Given two probability measures $\mu_1$ and $\mu_2$ on $\R^d$, we define the following quantity
  \begin{equation}\label{e:wdis}
  W_\psi(\mu_1,\mu_2)=\inf_{\Pi\in \mathscr{C}(\mu_1,\mu_2)}\int_{\R^d\times\R^d} \psi(|x-y|)\,d\Pi(x,y),
  \end{equation}
where $|\cdot|$ is the Euclidean norm and $\mathscr{C}(\mu_1,\mu_2)$ is the collection of measures on $\R^d\times\R^d$ having $\mu_1$ and $\mu_2$ as marginals. When $\psi$ is concave, the above definition gives rise to a \emph{Wasserstein distance} $W_\psi$ in the space $\mathscr P(\R^d)$ of probability measures $\mu$ on $\R^d$ such that $\int \psi(|z|)\,\mu(d z)<\infty$. If $\psi(r)=r$ for all $r\geq 0$, then $W_\psi$ is the standard $L^1$-Wasserstein distance (with respect to the Euclidean norm $|\cdot|$), which will be denoted by $W_1(\mu_1,\mu_2)$ throughout this paper. Another well-known example for $W_\psi$ is given by $\psi(r)=\I_{(0,\infty)}(r)$, which leads to the total variation distance $W_{\psi}(\mu_1,\mu_2)= \frac{1}{2} \|\mu_1-\mu_2\|_{{\rm Var}}.$

\subsection{Exponential convergence in Wasserstein-type distances}\label{suce1}
Throughout this paper, we denote by
  $$J(s):= \inf_{|x|=s} \big[\nu\wedge (\delta_x\ast \nu)\big]( \R^d),\quad s>0.$$
Condition \eqref{th10} implies that $  \inf_{0<s\le \kappa_0} J(s)>0$ for some $\kappa_0>0$.
The following result is the first main contribution of our paper on exponential convergence in the $L^1$-Wasserstein distance and the total variation for the SDE \eqref{s1}. Refer to Theorems \ref{thtpw} and \ref{thtvart} below for more general statements.

\begin{theorem}\label{th1} The following two assertions hold.
\begin{itemize}
\item[\rm(a)] Assume that there are constants $\alpha\in [0,1)$ and $\theta\in (0,\infty)$ such that
  \begin{equation}\label{th13}
  \lim_{r\to 0}\inf_{s\in (0,r]} J(s)s^\alpha \bigg(\log\frac1s\bigg)^{-1-\theta}>0.
  \end{equation} If the drift term $b$ satisfies ${\mathbf B(K_1 r^\beta, K_2 r,l_0)}$ with some constants $\beta\in [1-\alpha,1],\, K_1,l_0\ge0$ and $K_2>0$, then there exist constants $\lambda, c>0$ such that for any $x,y\in\R^d$ and $t>0$,
  \begin{equation}\label{th1-1}
  W_1(\delta_xP_t, \delta_yP_t)\le c e^{-\lambda t}|x-y|.
  \end{equation}
\item[\rm(b)] Assume that \eqref{th13} holds with $\alpha=0$, i.e. there is a constant $\theta\in (0,\infty)$ such that
  \begin{equation}\label{th13-11}
  \lim_{r\to 0}\inf_{s\in (0,r]} J(s)\bigg(\log\frac1s\bigg)^{-1-\theta}>0.
  \end{equation}
If the drift term $b$ satisfies ${\mathbf B(K_1, K_2 r,l_0)}$ with some constants  $K_1,l_0\ge0$ and $K_2>0$, then there exist constants $\lambda, c>0$ such that for any $x,y\in\R^d$ and $t>0$,
  \begin{equation}\label{thver3}
  \|\delta_xP_t- \delta_yP_t\|_{\rm Var}+W_1(\delta_xP_t, \delta_yP_t)\le c e^{-\lambda t}(1+|x-y|).
  \end{equation} In particular,  \begin{equation}\label{thsrvar-1}
  \|\delta_xP_t- \delta_yP_t\|_{\rm Var}\le c e^{-\lambda t}(1+|x-y|).
  \end{equation}
\end{itemize}
\end{theorem}

Let us make some comments on Theorem \ref{th1}. Firstly, by Example
\ref{exa} below, the condition \eqref{th13} is satisfied for any
(truncated) symmetric $\alpha'$-stable process with $\alpha'\in
(\alpha,2)$. In this case, the condition ${\mathbf B(K_1 r^\beta, K_2 r,l_0)}$ in
part (a) holds if the drift coefficient $b$ is dissipative outside some
compact set and $\beta$-H\"older continuous with $\beta\geq 1-\alpha$.
The latter is weaker than the assumptions on $b$ in
\cite[Corollary 1.4(i)]{CSZ}, which further implies that the SDE
\eqref{s1} has a unique strong solution.

\begin{example}\label{exa}\it Suppose that
 \begin{equation}\label{c:exa}\nu(dz)\ge \I_{\{0<z_1\le 1\}}\frac{c_{d,\alpha}}{|z|^{d+\alpha}}\,dz\end{equation}
for some $\alpha\in (0,2)$ and $c_{d,\alpha}>0$. Then, $J(s)\geq \tilde c_{d,\alpha}s^{-\alpha}$ for any $s>0$ small enough.
\end{example}

From Example \ref{exa} above, we can immediately get exponential convergence rates in the $L^1$-Wasserstein distance and the total variation for the SDE \eqref{s1}, when the L\'{e}vy noise $Z$ has a (truncated) $\alpha$-stable component for all $\alpha\in(0,2)$ and the drift term $b$ is dissipative outside some ball. Therefore, Theorem \ref{th1} covers the main result of \cite{Wang15} (see Theorem 1.2 therein). On the other hand, Example \ref{exa} indicates that Theorem \ref{th1} works for L\'evy processes whose associated L\'{e}vy measure does not necessarily have a rotationally invariant component. Therefore, Theorem \ref{th1} partially extends the framework of \cite{M15}. We note that, in order to derive \eqref{thver3} and \eqref{thsrvar-1}, the assumptions concerning the concentration of the L\'evy measure around zero (small jump activity) are weaker in \cite{M15} than those in the present paper. More precisely, \cite[Corollary 1.2]{M15} works even for finite measures (e.g.\ see \cite[Example 1.7]{M15}), while \eqref{th13-11} essentially requires that the L\'evy measure is infinite.

The approach of Theorem \ref{th1} is based on the coupling for L\'{e}vy processes, as in \cite{Wang15, M15}. It seems that the couplings used in \cite{Wang15, M15} depend heavily on the rotational symmetry of the L\'{e}vy measure, and so they do not work in our general setting, since we do not assume that the L\'{e}vy process $Z$ has a symmetric $\alpha$-stable component or the associated L\'{e}vy measure of $Z$ has a rotationally invariant absolutely continuous component. Therefore, some new ideas are required for the construction of the coupling.
One key ingredient of the proof in the paper relies, similarly to \cite{Eberle, Wang15, M15}, on using Wasserstein distances of type $W_\psi$ defined by \eqref{e:wdis} with appropriately chosen concave test functions $\psi$, which, in some sense, are comparable with $W_1$ for the estimate \eqref{th1-1}, or are intermediate between $W_1$ and the total variation for \eqref{thver3}.  It is worth pointing out that our choice of the concave test function $\psi$ satisfying $\psi(r)\asymp r$ (see Theorem \ref{thtpw} below) is quite simple.
The choice explicitly reflects the properties of the L\'evy measure $\nu$ and the drift $b$, and yields the explicit expression of $\lambda$ in \eqref{th1-1}, which is optimal in the sense that it is the same as that when $b$ satisfies the uniformly dissipative condition (see Remark \ref{r:thtpw}(1) below). The example below further indicates that, when $b$ satisfies ${\mathbf B(K_1r,K_2r, l_0)}$ and $Z$ is a symmetric $\alpha$-stable process, the constant $\lambda$ in \eqref{th1-1} is also of optimal order as $\alpha\to 2$.

\begin{example}\label{est-stable}\it Let $Z$ be a L\'evy process such that the associated L\'evy measure $\nu$ satisfies \eqref{c:exa} for some $\alpha\in (0,2)$. Assume that the drift $b$ satisfies ${\mathbf B(K_1r,K_2r, l_0)}$ for some constants $K_1, l_0\ge 0$ and $K_2>0$.
Then, there are constants $c_1,c_2>0$ such that the constant $\lambda$ in \eqref{th1-1} satisfies
$$
 \lambda\ge
  \begin{cases}
 c_1(K_2 \wedge l_0^{-\alpha}), & K_1l_0^{\alpha}\le 1;\\
 c_1(K_2\wedge l_0^{-\alpha})e^{-c_2K_1l_0^\alpha}, & K_1l_0^\alpha>1.
  \end{cases}$$
In particular, if $K_1=0$ and $l_0$ is large, then $\lambda\ge c_1 l_0^{-\alpha}$; if $K_1>0$ and $l_0$ is large, then $\lambda\ge c_1 l_0^{-\alpha} e^{-c_2 K_1l_0^{\alpha}}$.
Taking into account the related discussions in \cite[Section 2.3]{Eberle} for diffusions, we find that the lower bounds for $\lambda$ are of optimal orders with respect to $l_0, K_1$ and $K_2$ when $\alpha\to 2$ $($i.e.\ $Z$ is replaced by the standard Brownian motion$)$.
\end{example}

Secondly, we can see from Theorem \ref{th1} that there is a balance between the contributions of the noise and the drift in \eqref{s1}. On the one hand, in the uniformly dissipative case, one does not need assumptions like \eqref{th10} on the L\'evy measure, instead, a simple application of the synchronous coupling yields the exponential contractivity with respect to the $L^1$-Wasserstein distance. Note that our coupling \eqref{basic-coup-3} reduces to the synchronous coupling if $\kappa_0=0$ (see also the formula \eqref{proofth2544}). On the other hand, if the drift is locally non-dissipative and H\"older continuous, then the noise is required to fulfill the stronger condition \eqref{th13}.

Thirdly, assuming that the L\'{e}vy measure $\nu$ of $Z$ satisfies \eqref{th10} for some bounded and non-degenerate jumps (not the stronger condition \eqref{th13}) but has a rotationally invariant density function with respect to the Lebesgue measure, and the drift term $b$ satisfies ${\mathbf B(K_1 r, K_2 r,l_0)}$ with some constants $K_1,l_0\ge0$ and $K_2>0$, Majka \cite{M15} also proved \eqref{thver3}, see \cite[Assumptions 1-5 and Corollary 1.2]{M15} for more details. It is obvious that \eqref{thver3} does not imply \eqref{th1-1}.
We note that \eqref{th1-1} was also proved in \cite[Theorem 3.1]{M16} under an additional \lq\lq high concentration around zero\rq\rq \,\,assumption on the L\'evy measure, see \cite[Assumption L5]{M16}. As mentioned in \cite[Remark 1.6]{M15}, sufficient concentration of the L\'evy measure near zero (that is, the L\'evy noise enjoys a lot of small jumps and exhibits a diffusion-like type of behavior) seems to be necessary for obtaining \eqref{th1-1} rather than \eqref{thver3}, which can be obtained under much milder conditions. \eqref{th13} and \eqref{th13-11} as well as \cite[Assumption L5]{M16} are about sufficiently high small jump activity, and they all require the L\'evy measure to be infinite.  \eqref{th13} and \eqref{th13-11} indicate that there is sufficient overlap of small jumps for the L\'evy measure and its translation, while under \cite[Assumption L5]{M16} the small jumps corresponding to the first marginal of the L\'evy measure itself do not have finite moment. So, we believe that in general cases \eqref{th13} and \eqref{th13-11} are not comparable with \cite[Assumption L5]{M16}.  However, concerning symmetric $\alpha$-stable processes, it follows from Example \ref{exa} that both \eqref{th13} and \eqref{th13-11} hold true for any $\alpha\in (0,2)$, but \cite[Assumption L5]{M16} is satisfied only with $\alpha\in[1,2)$, see the remark below \cite[Assumption L5]{M16}.

Applying Theorem \ref{th1} and using some standard arguments (e.g.\ see \cite[Corollary 2]{Eberle} or \cite[Corollary 1.8]{LW}), we can also obtain that, under assumptions of Theorem \ref{th1} and the following additional condition
  $$ \int_{\{|z|\ge1\}}|z|\,\nu(dz)<\infty,  $$
there exist a unique invariant probability measure $\mu$, some constants $c,\lambda>0$ and a positive measurable function $c(x)$ such that
  \begin{equation}\label{re-ex}
  W_1(\delta_x P_t, \mu)\le c e^{-\lambda t} W_1(\delta_x , \mu),\quad x\in \R^d, t>0
  \end{equation}
and
  \begin{equation}\label{re-ex1}
  \|\delta_x P_t -\mu\|_{\rm Var}\le c(x)e^{-\lambda t},\quad x\in \R^d, t>0.
  \end{equation}
In the literature, \eqref{re-ex1} is called the exponential ergodicity for the process $(X_t)_{t\ge0}$. Note that from \eqref{thver3}, one can only obtain the exponential ergodicity with respect to $W_\psi$, where $\psi(r)=r+\I_{(0,\infty)}(r)$. In particular, one only has
  $$W_1(\delta_x P_t, \mu)\le  e^{-\lambda t} \big(c_1W_1(\delta_x , \mu)+c_2(x)\big),\quad x\in \R^d, t>0$$
for some positive constant $c_1>0$ and some positive measurable function $c_2(x)$, instead of \eqref{re-ex}. See \cite[Corollary 1.8]{M15} for more details. We emphasize that getting bounds of type \eqref{th1-1} instead of \eqref{thver3} is important in some applications. For example, Majka \cite{M16} (see also \cite[Remark 1.6]{M15}) showed how \eqref{th1-1} is used to obtain the so-called transportation inequalities, which characterize the concentration of measure phenomenon for solutions of SDEs of the form \eqref{s1}. By the dual representation of $W_1$ (see e.g. \cite[(5.10)]{Chen1}), \eqref{th1-1} implies that the associated semigroup $(P_t)_{t\ge0}$ maps ${\rm Lip}_b(\R^d)$ into itself, where ${\rm Lip}_b(\R^d)$ denotes the set of bounded Lipschitz functions on $\R^d$. Such property is useful in studying the existence of a
unique invariant probability for Markov semigroups, see \cite{KPS} and \cite[Section 2.2]{Wang10}.

\subsection{Strong ergodicity}

We are also interested in obtaining the exponential rate for total variation which is stronger than \eqref{thsrvar-1}; that is, we want to prove
  \begin{equation}\label{thm-strong-ergo-1}
  \|\delta_xP_t- \delta_yP_t\|_{\rm Var}\le c e^{-\lambda t},\quad x,y\in \R^d, t>0
  \end{equation}
for some positive constants $c$ and $\lambda$. Note that, compared with \eqref{thsrvar-1}, \eqref{thm-strong-ergo-1} is equivalent to
 $$W_\psi(\delta_xP_t, \delta_yP_t)\le \frac12 ce^{-\lambda t}\psi(|x-y|)$$ with $\psi(r)=\I_{(0,\infty)}(r)$, which enjoys the same form as those in \eqref{th1-1} and \eqref{thver3}.

 As shown by the result below, \eqref{thm-strong-ergo-1} can be established by imposing stronger dissipative condition on the drift term $b$ outside some compact set. See Theorem \ref{thtvarst} below for more general statement.

\begin{theorem}\label{thm-strong-ergo}
Assume that the drift term $b$ satisfies ${\mathbf B(K_1, \Phi_2(r),l_0)}$ with some constants  $K_1,l_0\ge0$ and some positive measurable function $\Phi_2$ such that $\Phi_2(r)$ is bounded from below for $r$ large enough, and
  \begin{equation}\label{erg-2}
  \int_{r_0}^\infty \frac{1}{\Phi_2(s)}\,ds <\infty\quad\textrm{ for some }r_0>0.
  \end{equation}
If \eqref{th13-11} holds, then there exist constants $\lambda, c>0$ such that for any $x,y\in\R^d$ and $t>0$, \eqref{thm-strong-ergo-1} holds true.
\end{theorem}

A typical example for \eqref{erg-2} is that $\Phi_2(s)=K_2 s^{1+\theta}$ for some $K_2,\theta>0$. In this case, the drift term $b$ satisfies that for any $x,y\in\R^d$ with $|x-y|\ge l_0$,
  $$\langle b(x)-b(y), x-y\rangle\le - K_2|x-y|^{2+\theta}.$$
For instance, $b(x)=\nabla V(x)$ with $V(x)=-|x|^{2+\theta}\, (\theta>0)$ satisfies the condition above, see \cite[Example 1.7]{LW} or \cite[Example 1.3]{Wang15}.

Next we will consider the strong ergodicity (with respect to the total variation) by making use of Theorem \ref{thm-strong-ergo}.
 We emphasize that, to the best of our knowledge, the proposition below is the first result concerning the strong ergodicity of SDEs with L\'evy jumps via the coupling approach. We also note that \eqref{thm-strong-ergo-1}, rather than \eqref{thsrvar-1}, is a key point to yield the strong ergodicity.

\begin{proposition}\label{p-ergo}
Suppose that the L\'{e}vy measure $\nu$ of the process $Z$ fulfills \eqref{th13-11} and that
  \begin{equation}\label{red}
  \int_{\{|z|\ge 1\}}\log(1+|z|)\,\nu(dz)<\infty.
  \end{equation}
If $b$ satisfies ${\mathbf B(K_1r, \Phi_2(r),l_0)}$ with some constants  $K_1,l_0\ge0$ and some positive measurable function $\Phi_2$ satisfying $\liminf\limits_{r\to\infty} \frac{\Phi_2(r)}{r}>0$ and \eqref{erg-2}, then the process $(X_t)_{t\ge0}$ is strongly ergodic, i.e. there exist a unique invariant probability measure $\mu$  and some constants $c,\lambda>0$ such that
  $$\|\delta_x P_t -\mu\|_{\rm Var}\le ce^{-\lambda t},\quad x\in \R^d, t>0.$$
\end{proposition}

The remainder of this paper is arranged as follows. In the next
section, we will present the refined basic coupling process for
L\'evy processes, which is interesting on its own. To reveal the new idea
behind this refined basic coupling, we begin with the construction
of coupling operator for L\'evy processes. Then we consider the
corresponding coupling operator for the SDE \eqref{s1}. In
particular, we directly prove that there exists a system of SDEs,
which is associated with this coupling operator and admits a unique
strong solution. Based on the coupling process constructed above,
general approaches via the coupling idea to exponential convergence
rates in Wasserstein distance for the SDE \eqref{s1} are presented
in Section \ref{genecop}. Proofs of all the results in Section
\ref{section1} are given in Section \ref{section3}. In Section 5 we first present
another application of the refined basic coupling for
L\'evy processes; namely, the regularity of the semigroup
$(P_t)_{t\ge0}$ associated to the SDE \eqref{s1} under the one-sided
Lipschitz condition. We also discuss in
Subsection \ref{5.3} some variations of the exponential convergence
in $L^1$-Wasserstein distance. Some properties related to
\eqref{th10} are given in the appendix.
Finally, we note that couplings of SDEs with multiplicative L\'evy noises were treated in the recent paper \cite{LW18}, where part of results above have been extended.

\section{Refined basic coupling for L\'evy processes}\label{sectionse2}

In this section we shall first construct a new coupling operator for pure jump L\'evy processes, and then find the corresponding SDE for the coupling process. The reason that we choose to begin with the construction of the coupling operator is that it clearly reveals the idea behind the coupling.

\subsection{Coupling operator for L\'evy processes}\label{coupling-op}

Recall that a $d$-dimensional \emph{pure jump L\'evy process} $Z = (Z_t)_{t\geq 0}$ is a stochastic process on $\R^d$ with $Z_0=0$, stationary and independent increments and c\`adl\`ag sample paths. Its finite-dimensional distributions are uniquely characterized by the characteristic exponent or the symbol of characteristic function $\Ee e^{i\scalar{\xi,Z_t}} = e^{-t\Phi_Z(\xi)}$ with
  $$ \Phi_Z(\xi) =\int\Bigl(1-e^{i\langle{\xi},{z}\rangle}+i\langle{\xi},{z}\rangle\I_{B(0,1)}(z)\Bigr)\,\nu(dz),$$
where $\nu$ is the L\'evy measure, i.e.\ a $\sigma$-finite measure on $(\R^d, \mathscr{B}(\R^d))$ such that $\nu(\{0\})=0$ and the integral $\int(1\wedge |z|^2)\,\nu(dz)<\infty$. Its infinitesimal generator acting on $C_b^2(\R^d)$ is given by
  \begin{equation}\label{proofth21}
  \begin{split}
  L_Z f(x)=&\int\!\!\Big(f(x+z)-f(x)-\langle\nabla f(x), z\rangle\I_{B(0,1)}(z)\Big)\,\nu(dz).
  \end{split}
  \end{equation}
Recall that an operator $\widetilde L_Z$ acting on $C_b^2(\R^d\times \R^d)$ is called a coupling of $L_Z$, if for any $f,g\in C_b^2(\R^d)$, setting $h(x,y)=f(x)+g(y)$ for all $x,y\in\R^d$, then we have
  \begin{equation}\label{marginality}
  \widetilde L_Z h(x,y)= L_Z f(x)+ L_Z g(y).
  \end{equation}
If the coupling operator $\widetilde L_Z$ generates a Markov process $(Z^1_t, Z^2_t)_{t\geq 0}$ on $\R^d \times\R^d$, then the latter is called a coupling process of $Z$. The coupling time is the first time that the two marginal processes $(Z_t^1)_{t\ge0}$ and $(Z_t^2)_{t\ge0}$ meet each other; that is, the stopping time $T=\inf\{t\geq 0: Z^1_t =Z^2_t\}$. If $T$ is almost surely finite, then the coupling is called successful. After the coupling time, we often let the two marginal processes move together.

We note that, what is need in applications, for example to estimate Wasserstein distance between distributions of SDEs as in Theorem \ref{th1}, is a coupling of two copies of the same process in the sense that it has two marginal processes with the same transition probabilities (or the same finite dimensional distributions) but possibly different initial conditions. Clearly, the condition \eqref{marginality} is not sufficient to guarantee this. One standard approach for this is to impose an additional assumption or to check the existence (but not necessarily unique) of solutions to the martingale problem associated with the coupling operator $\widetilde L_Z$. See e.g.\, \cite[Sections 2 and 3]{CL} and \cite[Section 2]{PW} for the diffusion case, and \cite[Section 3.1]{Wang10} and \cite[Section 2.2]{Wang15} for the L\'evy case. In the present paper, instead, we start from the assumption that the SDE \eqref{s1} has a unique strong solution, which enables us to prove the existence of a unique strong solution of some SDE on $\R^{2d}$ whose infinitesimal generator coincides with the coupling operator constructed below, see Propositions \ref{2-prop-1} and \ref{3-prop-2}. We also note that, by \cite[Theorem 1, p.\
2]{BLP} and \cite[Corollary 2.5]{KU}, for a large class of SDEs with jumps, if the strong solution exists uniquely, then the weak solution is also
unique, which in turn yields that the corresponding martingale problem is well posed. So, the approach via SDE to obtain the existence of coupling process associated with the coupling operator is stronger than the martingale problem used in aforementioned papers.

We first give the intuitive ideas that lead to the particular construction of our coupling. In the construction of a coupling process for pure jump L\'evy process $Z$, we often require the coupling time $T$ to be as small as possible, which provides better convergence speed. To this end, the natural idea is to make the two marginal processes jump to the same point with the biggest possible rate. This is exactly the meaning of the \emph{basic coupling} in \cite[Example 2.10]{Chen}. Here the biggest jump rate is the maximum common part of the jump intensities. In our setting, it takes the form $\mu_{y-x}(dz):=[\nu\wedge (\delta_{y-x} \ast\nu)](dz)$, where $x\neq y$ are the positions of the two marginal processes before the jump.

\begin{remark}
We claim that $\mu_x$ is a finite measure on $(\R^d, \mathscr{B}(\R^d))$ for any $x\neq0$. Indeed, for any $x,z\in\R^d$ with $x\neq0$ and $|z|\le |x|/2$,  $|z-x|\ge |x|-|z|\ge |x|/2$, which implies
  \begin{equation*}
  \int_{\{|z|\le |x|/2\}}(\delta_x\ast \nu)(dz)=\int_{\{|z|\le |x|/2\}}\nu(d(z-x)) \le \int_{\{|u|\ge |x|/2\}}\nu(du).
  \end{equation*}
Consequently,
 \begin{align*}
  \mu_x(\R^d)&=\int_{\{|z|\le |x|/2\}}\,\mu_x(dz)+\int_{\{|z|> |x|/2\}}\,\mu_x(dz)\\
  &\le\int_{\{|z|\le |x|/2\}} (\delta_x\ast \nu)(dz)+\int_{\{|z|>|x|/2\}}\,\nu(dz)  \le 2\int_{\{|z|\ge |x|/2\}}\,\nu(dz)<\infty.
  \end{align*}
\end{remark}

The operator corresponding to the basic coupling can be written as follows: for any $f\in C_b^2(\R^d \times \R^d)$,
  \begin{align*}
  \widetilde L_Z f(x,y)&= \int \!\!\Big(f(x+z, y+z+(x-y))-f(x,y)-\<\nabla_x f(x,y), z\>\I_{\{|z|\leq 1\}}\\
  &\hskip30pt -\<\nabla_y f(x,y), x-y+z\>\I_{\{|z+(x-y)|\leq 1\}} \Big)\,\mu_{y-x}(dz)\\
  &\quad + \int \!\!\Big(f(x+z, y)-f(x,y)-\<\nabla_x f(x,y), z\>\I_{\{|z|\leq 1\}}\Big)\,(\nu-\mu_{y-x}) (dz)\\
  &\quad + \int \!\!\Big(f(x, y+z)-f(x,y)-\<\nabla_y f(x,y), z\>\I_{\{|z|\leq 1\}}\Big)\,(\nu-\mu_{x-y}) (dz).
  \end{align*}
Here and in what follows, $\nabla_xh(x,y)$ and $\nabla_yh(x,y)$ are defined as the gradient of $h(x,y)$ with respect to $x$, $y\in \R^d$, respectively. The last two integrals are needed so that the marginality \eqref{marginality} of the coupling operator is satisfied. This can be seen by using the following crucial identity (see Corollary \ref{C:mea}):
  \begin{equation}\label{ee}
  \begin{split}
  \mu_{-x}(d(z-x)) &=(\delta_{x}*\mu_{-x})(dz)=  \big[\delta_{x}*\big(\nu\wedge (\delta_{-x}*\nu)\big)\big](dz)\\
  &=\big((\delta_{x}*\nu)\wedge \nu\big)(dz)= \mu_x(dz).
  \end{split}
  \end{equation}
This coupling can be illustrated as follows:
  \begin{equation}\label{basic-coup-1}
  (x,y)\longrightarrow
    \begin{cases}
    (x+z, y+z+(x-y)), & \mu_{y-x}(dz);\\
    (x+z, y), & (\nu - \mu_{y-x})(dz);\\
    (x, y+z), & (\nu - \mu_{x-y})(dz).
    \end{cases}
  \end{equation}
The first row of this coupling is quite good in applications, since the distance between the two marginals decreases from $|x-y|$ to $|(x+z)- (y+z+(x-y))|=0$. The second row, however, is not so welcome, because the new distance is $|x-y+z|$, which can be much bigger than the original one when the jump size $z$ is large. The same problem appears in the last row of the coupling.

Therefore, we have to modify the basic coupling to make it behave better. As a first step, we want to change the second row in \eqref{basic-coup-1} so that the distance after the jump is comparable with $|x-y|$. Inspired by the first row, a simple choice is $(x,y)\to (x+z, y+z+(y-x))$ with rate $\frac12 \mu_{x-y}(dz)$, where the distance after the jump is $2|x-y|$. The price to pay is that we need to modify at the same time the first row in \eqref{basic-coup-1}, so that the two marginal processes cannot jump to the same point with the biggest possible rate, but only half of it. For the last row, we simply let them jump with the same size and their distance remains unchanged. So the coupling \eqref{basic-coup-1} becomes
  \begin{equation}\label{basic-coup-2}
  (x,y)\longrightarrow
    \begin{cases}
    (x+z, y+z+(x-y)), & \frac12 \mu_{y-x}(dz);\\
    (x+z, y+z+(y-x)), & \frac12 \mu_{x-y}(dz);\\
    (x+z, y+z), & \big(\nu - \frac12 \mu_{y-x}  -\frac12 \mu_{x-y}\big)(dz).
    \end{cases}
  \end{equation}
Thanks to the identity \eqref{ee} again,
we are able to verify the marginality \eqref{marginality} for this modified coupling.

The above coupling \eqref{basic-coup-2} has a drawback too. If the original pure jump L\'evy process $Z$ is of finite range, then the jump intensity $\mu_{y-x}(dz)$ is identically zero for $|y-x|$ large enough. Thus the two marginal processes of the coupling \eqref{basic-coup-2} will never get closer if they are initially far away. Our intuitive idea to overcome this difficulty is that if the distance between the marginal processes is already small, then we let them jump as in \eqref{basic-coup-2}; while if the distance  is too large, then it would be more reasonable to reduce it by a small amount after each jump, since the requirement that their distance decreases to zero seems too greedy. Thus, we introduce a parameter $\kappa>0$ which serves as the threshold to determine whether the marginal processes jump to the same point or become slightly closer to each other. Let $\kappa_0$ be the constant in \eqref{th10}. For any $x$, $y\in\R^d$ and $\kappa\in(0, \kappa_0]$, define
  \begin{equation}\label{e:xy}
  (x-y)_{\kappa}=\bigg(1\wedge \frac{\kappa}{|x-y|}\bigg)(x-y).
  \end{equation}
We make the convention that $(x-x)_\kappa=0$. Then our coupling is given as follows:
  \begin{equation}\label{basic-coup-3}
  (x,y)\longrightarrow
    \begin{cases}
    (x+z, y+z+(x-y)_\kappa), & \frac12 \mu_{(y-x)_\kappa}(dz);\\
    (x+z, y+z+(y-x)_\kappa), & \frac12 \mu_{(x-y)_\kappa}(dz);\\
    (x+z, y+z), & \big(\nu - \frac12 \mu_{(y-x)_\kappa}  -\frac12 \mu_{(x-y)_\kappa}\big)(dz).
    \end{cases}
  \end{equation}
We see that if $|x-y|\leq \kappa$, then the above coupling is the same as that in \eqref{basic-coup-2}. If $|x-y|>\kappa$, then according to the first two rows, the distances after the jump are $|x-y|-\kappa$ and $|x-y|+\kappa$, respectively. We will call the coupling given by \eqref{basic-coup-3} the \emph{refined basic coupling} for pure jump L\'evy processes.

We make some further comments on the construction of the refined basic coupling. We first note that this construction does not require any geometric assumption on the L\'evy measure. Second, in order to obtain a coupling with good optimality properties, it is well known from the theory of optimal transport that one should not remove the common mass of two probability distributions, so in this sense the first row in \eqref{basic-coup-1} is natural (see \cite[Section 2.1]{M15} for more details). Therefore, the question is what one should do with the remaining mass. If the L\'evy measure is rotationally invariant, Majka \cite[Section 2.2]{M15} applied reflection to the remaining mass. For general setting, one can try to apply the independent coupling to the remaining mass as indicated in \eqref{basic-coup-1}. However, as mentioned in remarks below \eqref{basic-coup-1}, such coupling does not behave well. Intuitively, a much better solution would be to couple the remaining mass synchronously, but it turns out that such a construction does not produce a coupling.  In the preliminary construction of the refined basic coupling \eqref{basic-coup-2}, we send the two marginal processes to the same place only with half of the maximal probability (see the first row in \eqref{basic-coup-2}), and with the other half  we perform  a transformation which doubles the distance between the  two marginal processes  (see the second row in \eqref{basic-coup-2}). With this transformation, we can apply the synchronous movement with the remaining probability (see the third row in \eqref{basic-coup-2}) and still obtain a coupling. From the refined basic coupling constructed above, it seems that in some cases it may be a good idea to give up jumping to the same place with the maximal possible probability, since decreasing that probability may allow us to coupling the remaining mass in a more convenient way. Such an idea would be helpful in the study of constructing couplings of non-symmetric L\'evy processes with good optimality properties, which seems to be an interesting open problem. The readers can refer to related discussions in the end of \cite[Section 5]{B17}, where Makovian maximal coupling for subordinated Brownian motions, partly motivated by \cite{BSW}, was investigated.

We can now write explicitly the coupling operator $\widetilde L_Z$ corresponding to \eqref{basic-coup-3}. Fix $h\in C_b^2(\R^{2d})$. For any $x,y\in\R^d$, we define
  \begin{align}\label{proofth24}
    \widetilde{L}_Z h(x,y)
    &= \frac{1}{2}\int\Big( h(x+z,y+ z+(x-y)_{\kappa})-h(x,y)-\langle\nabla_xh(x,y), z\rangle \I_{\{|z|\le 1\}}\cr
    &\qquad\qquad -\langle\nabla_yh(x,y), z+(x-y)_{\kappa}\rangle \I_{\{|z+(x-y)_{\kappa}|\le 1\}}\Big)\,\mu_{(y-x)_{\kappa}}(dz)\cr
    &\quad +\frac{1}{2}\int\Big( h(x+z,y+ z+(y-x)_{\kappa})-h(x,y)-\langle\nabla_xh(x,y), z\rangle \I_{\{|z|\le 1\}}\\
    &\qquad\qquad -\langle\nabla_yh(x,y), z+(y-x)_{\kappa}\rangle \I_{\{|z+(y-x)_{\kappa}|\le 1\}}\Big)\,\mu_{(x-y)_{\kappa}}(dz)\cr
    &\quad +\int\Big( h(x+z,y+z)-h(x,y)-\langle\nabla_xh(x,y), z\rangle \I_{\{|z|\le 1\}}\cr
    &\qquad\qquad-\langle\nabla_yh(x,y), z\rangle \I_{\{|z|\le 1\}}\Big)\,\Big(\nu -\frac{1}{2}\mu_{(x-y)_{\kappa}} -\frac{1}{2}\mu_{(y-x)_{\kappa}}\Big)(dz). \nonumber
  \end{align}

Below, we prove rigorously that $\widetilde{L}_Z$ is indeed a coupling operator of the operator $L_Z$ given by \eqref{proofth21}. For this we let $h(x,y)=g(y)$ for any $x,y\in\R^d$, where $g\in C_b^2(\R^d)$. Then, according to \eqref{proofth24},
  \begin{align*}
    \widetilde{L}_Z h(x,y) &= \frac{1}{2}\int \Big(g(y+ z+(x-y)_{\kappa})-g(y)\\
    &\qquad\qquad -\langle\nabla g(y), z+(x-y)_{\kappa}\rangle \I_{\{|z+(x-y)_{\kappa}|\le 1\}}\Big)\,\mu_{(y-x)_{\kappa}}(dz)\\
    &\quad+\frac{1}{2}\int \Big( g(y+ z+(y-x)_{\kappa})-g(y)\\
    &\qquad\qquad -\langle\nabla g(y), z+(y-x)_{\kappa}\rangle \I_{\{|z+(y-x)_{\kappa}|\le 1\}}\Big)\,\mu_{(x-y)_{\kappa}}(dz)\\
    &\quad+\int \Big(g(y+z)-g(y)-\langle\nabla g(y), z\rangle \I_{\{|z|\le 1\}}\Big) \Big(\nu -\frac{1}{2}\mu_{(x-y)_{\kappa}} -\frac{1}{2}\mu_{(y-x)_{\kappa}}\Big)(dz).
    \end{align*}
Changing the variables $z+(x-y)_{\kappa}\to u$ and $z+(y-x)_{\kappa} \to u$ respectively leads to
  \begin{align*}
    \widetilde{L}_Z h(x,y)&=\frac{1}{2}\int\Big(g(y+u)-g(y)-\langle\nabla g(y), u\rangle \I_{\{|u|\le 1\}}\Big)\,\mu_{(y-x)_{\kappa}}(d(u-(x-y)_{\kappa}))\\
    &\quad+\frac{1}{2}\int\Big( g(y+u)-g(y)-\langle\nabla g(y), u\rangle \I_{\{|u|\le 1\}}\Big)\,\mu_{(x-y)_{\kappa}}(d(u-(y-x)_{\kappa}))\\
    &\quad+\int\Big(g(y+z)-g(y)-\langle\nabla g(y), z\rangle \I_{\{|z|\le 1\}}\Big) \Big(\nu -\frac{1}{2}\mu_{(x-y)_{\kappa}} -\frac{1}{2}\mu_{(y-x)_{\kappa}}\Big)(dz).
    \end{align*}
By \eqref{ee}, the expression above is equal to $L_Zg(y)$, cf.\ \eqref{proofth21}. Thus,
we can easily conclude that the operator $\widetilde L_Z$ defined by \eqref{proofth24} is a coupling operator of $L_Z$, i.e. \eqref{marginality} holds.

The existence of Markov processes associated with the coupling operator $\widetilde L_Z$ defined by \eqref{proofth24} will be  proved in the next subsection via the SDE approach. More explicitly, according to Propositions \ref{2-prop-1}, \ref{3-prop-2} and Remark \ref{r-coulevy} below, we can find  a Markov process $(Z_t, Z^\ast_t)_{t\ge0}$ on $\R^{2d}$, as a unique strong solution for some SDE (see \eqref{coup-SDE-1} and \eqref{Z-star} below), such that the associated infinitesimal generator is exactly the coupling operator $\widetilde L_Z$.

\subsection{Coupling process for L\'evy processes}\label{cou:pro}

The aim of this subsection is to find the SDE associated with the coupling operator $\widetilde L_Z$ defined above. This will help us with constructing the coupling process by solving the SDE.

For a pure jump L\'{e}vy process $Z$, by the L\'{e}vy--It\^{o} decomposition, there exists a Poisson random measure $N(ds,dz)$ associated with $Z$ in such a way that
  $$Z_t=\int_0^t\int_{\{|z|>1\}}z\,N(ds,dz)+\int_0^t\int_{\{|z|\le1\}}z\,\tilde{N}(ds,dz),$$
where
  $$\tilde{N}(ds,dz)=N(ds,dz)-ds\,\nu(dz)$$
is the compensated Poisson measure. Recall that there exist a sequence of random variables $(\tau_j)_{j\ge1}$ in $\R_+$ encoding the jump times and a sequence of random variables $(\xi_j)_{j\ge1}$ in $\R^d$ encoding the jump sizes such that
  $$N((0,t],A)(\omega)=\sum_{j=1}^\infty \delta_{(\tau_j(\omega), \xi_j(\omega))}((0,t]\times A),\quad \omega\in \Omega, A\in \mathscr{B}(\R^d).$$
To construct a coupling process, let us follow the idea in \cite[Section 2.2]{M15} and begin with extending the Poisson random measure $N$ on $\R_+\times \R^d$ to a Poisson random measure on $\R_+\times \R^d\times [0,1]$, by replacing the $d$-dimensional random variables $\xi_j$ determining the jump sizes of $(Z_t)_{t\ge0}$ with the $(d+1)$-dimensional random variables $(\xi_j, \eta_j)$, where each $\eta_j$ is a uniformly distributed random variable on $[0,1]$. Thus, we have
  \begin{equation*}\label{Poisson-meas}
  N((0,t],A)(\omega)=\sum_{j=1}^\infty \delta_{(\tau_j(\omega), \xi_j(\omega), \eta_j(\omega))}((0,t]\times A\times [0,1]),\quad \omega\in \Omega, A\in \mathscr{B}(\R^d).
  \end{equation*}
To save notations, we still denote the extended Poisson random measure by $N$, and write
  $$Z_t=\int_0^t\int_{\{|z|>1\}\times [0,1]}z\,N(ds,dz,du)+\int_0^t\int_{\{|z|\le1\}\times[0,1]}z\,\tilde{N}(ds,dz,du).$$
For simplicity, we set
  \begin{equation}\label{Poisson-random-measure}
  \bar{N}(ds,dz,du)=\I_{\{|z|>1\}\times[0,1]}N(ds,dz,du)+\I_{\{|z|\le1\}\times [0,1]}\tilde{N}(ds,dz,du)
  \end{equation}
and hence
  $$Z_t=\int_0^t\int_{\R^d\times[0,1]} z\,\bar{N}(ds,dz,du).$$
or equivalently,
  \begin{equation}\label{coup-SDE-1}
  d Z_t= \int_{\R^d\times[0,1]} z\,\bar{N}(dt,dz,du).
  \end{equation}
We want to find the SDE for the process $Z^\ast:=(Z^\ast_t)_{t\geq 0}$ so that $(Z_t, Z^\ast_t)_{t\geq 0}$ is a Markov process on $\R^{2d}$, and has the coupling operator $\widetilde L_Z$ constructed in \eqref{proofth24} as its generator.

With the above notations and taking into account the construction \eqref{basic-coup-3} of the coupling operator $\widetilde L_Z$, if a jump occurs at time $t$, then the process $Z$ moves from the point $Z_{t-}$ to $Z_{t-} + z$, and we draw a random number $u\in [0,1]$ to determine whether the process $Z^\ast$ should jump from the point $Z^\ast_{t-}$ to the points $Z^\ast_{t-} +z +(Z_{t-} -Z^\ast_{t-})_\kappa$, $Z^\ast_{t-}+z +(Z^\ast_{t-} -Z_{t-})_\kappa$ and $Z^\ast_{t-}+z$, respectively. To this end, we define the control function $\rho$ as follows: for any $x,z\in\R^d$,
  $$\rho(x,z)=\frac{\nu \wedge (\delta_x\ast \nu) (dz)}{\nu(dz)}\in [0,1]. $$
By convention, $\rho(0,z)\equiv 1$ for all $z\in\R^d$. For simplification of notations, we write $U_t=Z_{t} -Z^\ast_{t}$ and consider the following SDE:
  \begin{equation}\label{coup-SDE-2}
  \begin{split}
  d Z^\ast_t&= \int_{\R^d\times [0,1]} \Big[\big(z+(U_{t-})_{\kappa}\big) \I_{\{u\le \frac12 \rho((-U_{t-})_{\kappa},z)\}} \\
  &\quad \quad +\big(z+(-U_{t-})_{\kappa}\big) \I_{\{\frac12 \rho((-U_{t-})_{\kappa},z)< u\le \frac12 [\rho((-U_{t-})_{\kappa},z)+\rho((U_{t-})_{\kappa},z)]\}}\\
  &\quad\quad  + z \I_{\{\frac12 [\rho((-U_{t-})_{\kappa},z)+\rho((U_{t-})_{\kappa},z)]< u\le 1\}}\Big]  \bar{N}(dt,dz,du)\\
  &\quad - \int_{\R^d\times [0,1]} \!
  \Big[ \big(z+(U_{t-})_{\kappa}\big)\! \big(\I_{\{|z+(U_{t-})_{\kappa}|\le 1\}} \! -\!\I_{\{|z|\le 1\}}\big)\! \I_{\{u\le \frac12 \rho((-U_{t-})_{\kappa},z)\}}\\
  &\quad \quad +\big(z+(-U_{t-})_{\kappa}\big) \big(\I_{\{|z+(-U_{t-})_{\kappa}|\le 1\}} -\I_{\{|z|\le 1\}}\big)\\
  &\quad \quad \quad \times \I_{\{\frac12 \rho((-U_{t-})_{\kappa},z)< u\le
  \frac12 [\rho((-U_{t-})_{\kappa},z)+\rho((U_{t-})_{\kappa},z)]\}}\Big] \,\nu(dz)\,du\,dt.
  \end{split}
  \end{equation}
Here, the first integral with respect to the Poisson random measure corresponds to three jumps in \eqref{basic-coup-3}, while the second integral is needed to ensure that $(Z_t, Z^\ast_t)_{t\geq 0}$ has the generator $\widetilde L_Z$, see the proof of Proposition \ref{3-prop-2} below.

The equation \eqref{coup-SDE-2} looks a little complicated, thus we have to simplify it before moving forward. Recall that for $x,y\in\R^d$ and $\kappa\in (0,\kappa_0]$, $(x-y)_{\kappa}$ is given by \eqref{e:xy}. By collecting the terms involving $z$, we can rewrite the above equation as
  \begin{equation}\label{e:couppro}
  \begin{split}
  d{Z}^\ast_t&=\int_{\R^d\times [0,1]} z \, \bar{N}(dt,dz,du)\\
  &\quad + \int_{\R^d\times [0,1]}\Big[(U_{t-})_{\kappa} \I_{\{u\le \frac12 \rho((-U_{t-})_{\kappa},z)\}}\\
  &\quad\quad +(-U_{t-})_{\kappa} \I_{\{\frac12 \rho((-U_{t-})_{\kappa},z)< u\le \frac12 [\rho((-U_{t-})_{\kappa},z)+\rho((U_{t-})_{\kappa},z)]\}}\Big]
  \bar{N}(dt,dz,du)\\
  &\quad -\frac12 \int_{\R^d}
  \Big[\big(z+(U_{t-})_{\kappa}\big) \big(\I_{\{|z+(U_{t-})_{\kappa}|\le 1\}} -\I_{\{|z|\le 1\}}\big) \rho((-U_{t-})_{\kappa},z)\\
  &\quad\quad +\big(z+(-U_{t-})_{\kappa}\big) \big(\I_{\{|z+(-U_{t-})_{\kappa}|\le 1\}} -\I_{\{|z|\le 1\}}\big)  \rho((U_{t-})_{\kappa},z)\Big] \,\nu(dz) \,dt.
  \end{split}
  \end{equation}

Observe that if $U_{t-}=Z_{t-}-Z^\ast_{t-}=0$, then $dZ_t^\ast=dZ_t$; if $U_{t-}\neq0,$ then, by the fact that $\mu_x=\nu\wedge (\delta_x\ast \nu)$ is a finite measure on $(\R^d, \mathscr{B}(\R^d))$ for any $x\neq0$,
  \begin{equation}\label{prop-2.1}
  \int_{\R^d \times[0,1]}  \I_{\{ u\le \frac{1}{2}\rho((-U_{t-})_{\kappa},z)\}}\,\nu(dz)\,du =\frac12\mu_{(-U_{t-})_\kappa}(\R^d) <\infty
  \end{equation}
and
  \begin{equation}\label{prop-2.2}
  \begin{split}
  &\int_{\R^d \times[0,1]} \!\!\!\! \I_{\{\frac{1}{2}\rho((-U_{t-})_{\kappa},z)<u\le \frac{1}{2}[\rho((-U_{t-})_{\kappa},z)+\rho((U_{t-})_{\kappa},z)]\}}\nu(dz)du
 \!=\! \frac12 \mu_{(U_{t-})_\kappa}\!(\R^d) \!<\infty.
  \end{split}
  \end{equation}
Hence,
  $$\int_{\R^d \times[0,1]}  \I_{\{ u\le \frac{1}{2}\rho((-U_{t-})_{\kappa},z)\}}\,\bar{N}(dt,dz,du)$$
and
  $$\int_{\R^d \times[0,1]}  \I_{\{\frac{1}{2}\rho((-U_{t-})_{\kappa},z)<u\le \frac{1}{2}[\rho((-U_{t-})_{\kappa},z)+\rho((U_{t-})_{\kappa},z)]\}}\,\bar{N}(dt,dz,du)$$
are well defined.

We denote by $J_i\, (1\leq i \leq 3)$ the three terms on the right hand side of \eqref{e:couppro}. On the one hand, using \eqref{ee} and changing variable $z+(U_{t-})_\kappa \to z$ lead to
  \begin{align*}&\frac12 \int_{\R^d}
  \big(z+(U_{t-})_{\kappa}\big) \big(\I_{\{|z+(U_{t-})_{\kappa}|\le 1\}} -\I_{\{|z|\le 1\}}\big) \rho((-U_{t-})_{\kappa},z)\,\nu(dz)\\
  &= \frac12\int_{\R^d} z \big(\I_{\{|z|\le 1\}} -\I_{\{|z+(-U_{t-})_{\kappa}|\le 1\}}\big)  \rho((U_{t-})_{\kappa},z)\,\nu(dz).
  \end{align*}
Thus,
  $$J_3=\frac12 (U_{t-})_{\kappa}\int_{\R^d} \big(\I_{\{|z+(-U_{t-})_{\kappa}|\le 1\}} -\I_{\{|z|\le 1\}}\big)  \rho((U_{t-})_{\kappa},z) \,\nu(dz) \,dt.$$
On the other hand, the subtracted term in the martingale part of $J_2$ is
  \begin{align*}&\int_{\{|z|\le 1\} \times [0,1]}\Big[(U_{t-})_{\kappa} \I_{\{u\le \frac12 \rho((-U_{t-})_{\kappa},z)\}}\\
  &\quad\qquad\qquad +(-U_{t-})_{\kappa} \I_{\{\frac12 \rho((-U_{t-})_{\kappa},z)< u\le \frac12 [\rho((-U_{t-})_{\kappa},z)+\rho((U_{t-})_{\kappa},z)]\}}\Big]
  \,\nu(dz)\,du\, dt\\
  &=\frac12 (U_{t-})_\kappa\bigg[\int_{\{|z|\le 1\}}\rho((-U_{t-})_{\kappa},z)\,\nu(dz)-\int_{\{|z|\le 1\}}\rho((U_{t-})_{\kappa},z)\,\nu(dz)\bigg]dt  \\
  &=\frac12 (U_{t-})_\kappa \int \big(\I_{\{|z+(-U_{t-})_{\kappa}|\le 1\}} -\I_{\{|z|\le 1\}}\big)  \rho((U_{t-})_{\kappa},z) \,\nu(dz)\,dt,
  \end{align*}
where in the last equality we also used \eqref{ee}. According to both equalities above, we can write \eqref{e:couppro} in an equivalent but more convenient way as
  \begin{equation*}
  \begin{split}
  d{Z}^\ast_t&=\int_{\R^d\times [0,1]} z \, \bar{N}(dt,dz,du) + (U_{t-})_{\kappa} \int_{\R^d\times [0,1]}\Big[ \I_{\{u\le \frac12 \rho((-U_{t-})_{\kappa},z)\}} \\
  &\qquad\qquad\qquad - \I_{\{\frac12 \rho((-U_{t-})_{\kappa},z)< u\le \frac12 [\rho((-U_{t-})_{\kappa},z)+\rho((U_{t-})_{\kappa},z)]\}}\Big]
  {N}(dt,dz,du).
  \end{split}
  \end{equation*}

We denote by
  $$V_{t}(z,u)=(U_{t})_{\kappa} \Big[ \I_{\{u\le \frac12 \rho((-U_{t})_{\kappa},z)\}} - \I_{\{\frac12 \rho((-U_{t})_{\kappa},z)< u\le \frac12 [\rho((-U_{t})_{\kappa},z)+\rho((U_{t})_{\kappa},z)]\}}\Big] $$
and
  $$dL^\ast_t=\int_{\R^d\times [0,1]}V_{t-}(z,u) \,{N}(dt,dz,du).$$
Then \eqref{coup-SDE-2} reduces to
  \begin{equation}\label{Z-star}
  d{Z}^\ast_t= dZ_t+dL^\ast_t.
  \end{equation}

By Remark \ref{r-coulevy} below, the process $(Z_t, Z^\ast_t)_{t\ge0}$ constructed above is a Markov coupling process for the L\'evy process $Z$, and its infinitesimal generator is $\widetilde L_Z$ defined in \eqref{proofth24}. Since the proof is similar to that of the coupling for the SDE \eqref{s1}, we postpone it in the next subsection.

\subsection{Coupling for the SDE \eqref{s1}}\label{sectioncousde}

In this part we study the coupling process of the solution $(X_t)_{t\geq 0}$ to the SDE \eqref{s1}. The infinitesimal generator of  $(X_t)_{t\geq 0}$ is
  \begin{equation}\label{SDE-generator}
  \begin{split}
  L_X f(x)&= \int\!\!\Big(f(x+z)-f(x)-\langle\nabla f(x), z\rangle\I_{\{|z|\leq 1\}} \Big)\,\nu(dz)+\langle b(x), \nabla f(x)\rangle\\
  &= L_Zf(x)+\langle b(x),  \nabla f(x)\rangle.
  \end{split}
  \end{equation}
Given the coupling operator $\widetilde L_Z$ in \eqref{proofth24} for the pure jump L\'evy process $Z$, it is natural to define $\widetilde L_X$ as follows: for any $h\in C_b^2(\R^d\times \R^d)$,
  \begin{equation}\label{SDE-coup-op}
  \widetilde L_X h(x,y)= \widetilde L_Z h(x,y) +\<b(x), \nabla_x h(x,y)\>+ \<b(y), \nabla_y h(x,y)\>.
  \end{equation}
Since $\widetilde L_Z$ is a coupling operator of $L_Z$, it is easy to see that $\widetilde{L}_X$ is a coupling operator of $L_X$ too.

Next we present the coupling equation corresponding to $\widetilde L_X$. Recall that the process $(X_t)_{t\geq 0}$ is generated by the SDE
  $$dX_t= b(X_t)\, dt+ dZ_t, \quad X_0=x.$$
Therefore, taking into account the equation \eqref{Z-star}, we denote by $U_t=X_t-Y_t$ and
  $$V_{t}(z,u)=(U_{t})_{\kappa} \Big[ \I_{\{u\le \frac12 \rho((-U_{t})_{\kappa},z)\}} - \I_{\{\frac12 \rho((-U_{t})_{\kappa},z)< u\le \frac12 [\rho((-U_{t})_{\kappa},z)+\rho((U_{t})_{\kappa},z)] \}}\Big]$$
for $z\in\R^d$ and $u\in[0,1]$. Then the marginal process $(Y_t)_{t\geq 0}$ of the coupling process $(X_t, Y_t)_{t\ge0}$ should fulfill the equation
  $$dY_t=b(Y_t)\, dt +d Z^\ast_t$$
with $d Z^\ast_t= dZ_t + dL^\ast_t$, where
  \begin{equation}\label{Z-star-SDE}
  d L^\ast_t= \int_{\R^d\times [0,1]}V_{t-}(z,u) \,{N}(dt,dz,du).
  \end{equation}

Fix any $x,y\in \R^d$ with $x\neq y$. We consider the system of equations:
  \begin{equation}\label{SDE-coup-eq-1}
  \begin{cases}
  dX_t=b(X_t)\, dt+dZ_t,& X_0=x;\\
  dY_t= b(Y_t)\,dt+ dZ_t+dL^\ast_t, & Y_0=y.
  \end{cases}
  \end{equation}

\begin{proposition}\label{2-prop-1}
The system of equations \eqref{SDE-coup-eq-1} has a unique strong solution.
\end{proposition}

\begin{proof}
In the setting of our paper, we always assume that the equation \eqref{s1} (i.e., the first equation in \eqref{SDE-coup-eq-1}) has a non-explosive and pathwise unique strong solution $(X_t)_{t\geq 0}$. We show that the sample paths of $(Y_t)_{t\geq 0}$ can be obtained by repeatedly modifying those of the solution of the following equation:
  \begin{equation}\label{2-prop-1.1}
  d \tilde Y_t=b(\tilde Y_t)\,dt+ dZ_t,\quad \tilde Y_0=y.
  \end{equation}

Denote by $Y^{(1)}_t$ the solution to \eqref{2-prop-1.1}. Take a uniformly distributed random variable $\zeta_1$ on $[0,1]$, and define the stopping times $T_1=\inf\big\{t>0: X_t=Y^{(1)}_t \big\}$ and
  \begin{align*}
  \sigma_1=\inf\bigg\{t>0: &\, \zeta_1\leq \frac12\Big[\rho\big((Y^{(1)}_t-X_t)_\kappa, \Delta Z_t\big)+ \rho\big((X_t-Y^{(1)}_t)_\kappa,\Delta Z_t\big)\Big]\bigg\}.
  \end{align*}
We consider two cases:
\begin{itemize}
\item[(i)] On the event $\{T_1\leq \sigma_1\}$, we set $Y_t=Y^{(1)}_t$ for all $t< T_1$; moreover, by the pathwise uniqueness of the equation \eqref{s1}, we can define $Y_t=X_t$ for $t\geq T_1$.
\item[(ii)] On the event $\{T_1> \sigma_1\}$, we define $Y_t=Y^{(1)}_t$ for all $t< \sigma_1$ and
  $$Y_{\sigma_1}=Y^{(1)}_{\sigma_1-}+ \Delta Z_{\sigma_1} + \begin{cases}
 \big(X_{\sigma_1 -} -Y^{(1)}_{\sigma_1 -}\big)_\kappa, & \mbox{if } \zeta_1\leq \frac12 \rho\big(\big(Y^{(1)}_{\sigma_1 -}-X_{\sigma_1 -}\big)_\kappa,\Delta Z_{\sigma_1} \big); \\
  \big( Y^{(1)}_{\sigma_1 -} -X_{\sigma_1 -}\big)_\kappa, & \mbox{if } \zeta_1> \frac12 \rho\big(\big(Y^{(1)}_{\sigma_1 -}-X_{\sigma_1 -}\big)_\kappa,\Delta Z_{\sigma_1} \big).
  \end{cases}$$
\end{itemize}

Next, we restrict on the event $\{T_1> \sigma_1\}$ and consider the SDE \eqref{2-prop-1.1} with $t>\sigma_1$ and $\tilde Y_{\sigma_1}=Y_{\sigma_1}$. Denote its solution by $Y^{(2)}_t$. Similarly, we take another uniformly distributed random variable $\zeta_2$ on $[0,1]$, and define $T_2=\inf\big\{t>\sigma_1: X_t=Y^{(2)}_t\big\}$ and
  \begin{align*}
  \sigma_2=\inf\bigg\{t>\sigma_1: &\, \zeta_2\leq \frac12\Big[\rho\big((Y^{(2)}_t-X_t)_\kappa, \Delta Z_t\big)+ \rho\big((X_t-Y^{(2)}_t)_\kappa,\Delta Z_t\big)\Big]\bigg\}.
  \end{align*}
In the same way, we can define the process $Y_t$ for $t\leq \sigma_2$. We repeat this procedure and note that, thanks to \eqref{prop-2.1} and \eqref{prop-2.2}, only finite many modifications have to be made in any finite interval of time. Finally, we obtain the sample paths $(Y_t)_{t\geq 0}$.
\end{proof}

Furthermore, the following conclusion indicates that the process $(X_t, Y_t)_{t\geq 0}$ is indeed the coupling process of $(X_t)_{t\ge0}$.

\begin{proposition}\label{3-prop-2}
The infinitesimal generator of the process $(X_t, Y_t)_{t\geq 0}$ is $\widetilde L_X$ defined in \eqref{SDE-coup-op}.
\end{proposition}

\begin{proof} According to the discussions in the previous subsection, the driven noise $(Z^\ast_t)_{t\ge0}$ defined by $Z^\ast_t= Z_t + L^\ast_t$ in the second equation of  \eqref{SDE-coup-eq-1} also enjoys the expression \eqref{coup-SDE-2} with $U_t=X_t-Y_t$ replacing $U_t=Z_t-Z^\ast_t$.
Then, the desired assertion can be proved by making use of the equations \eqref{coup-SDE-2} and \eqref{SDE-coup-eq-1} and applying the It\^{o} formula.  Indeed, denote by $\bar L_X$ the generator corresponding to $(X_t, Y_t)_{t\geq 0}$. For $h\in C_b^2(\R^d\times \R^d)$, by \eqref{coup-SDE-2} and \eqref{SDE-coup-eq-1}, we have
  \begin{align*}
  \bar L_X h(x,y)
  &= \frac12\int_{\R^d} \Big(h(x+z,y+ z+(x-y)_\kappa) -h(x,y) -\<\nabla_x h(x,y),z\> \I_{\{|z|\leq 1\}}\\
  &\hskip60pt -\<\nabla_y h(x,y),z+(x-y)_\kappa\> \I_{\{|z|\leq 1\}} \Big) \mu_{(y-x)_\kappa}(dz)\\
  &\hskip14pt +\frac12\int_{\R^d} \Big(h(x+z,y+ z+(y-x)_\kappa) -h(x,y) -\<\nabla_x h(x,y),z\> \I_{\{|z|\leq 1\}}\\
  &\hskip60pt -\<\nabla_y h(x,y),z+(y-x)_\kappa\> \I_{\{|z|\leq 1\}} \Big) \mu_{(x-y)_\kappa}(dz)\\
  &\hskip14pt +\int_{\R^d} \Big(h(x+z,y+ z) -h(x,y) -\<\nabla_x h(x,y),z\> \I_{\{|z|\leq 1\}} \\
  &\hskip60pt -\<\nabla_y h(x,y),z\> \I_{\{|z|\leq 1\}} \Big)\Big( \nu(dz)-\frac12 \mu_{(y-x)_\kappa}(dz) -\frac12 \mu_{(x-y)_\kappa}(dz)\Big)\\
  &\hskip14pt -\frac12 \int_{\R^d} \<\nabla_y h(x,y),z+(x-y)_\kappa\> \big(\I_{\{|z+(x-y)_\kappa|\leq 1\}} -\I_{\{|z|\leq 1\}} \big) \mu_{(y-x)_\kappa}(dz)\\
  &\hskip14pt -\frac12 \int_{\R^d} \<\nabla_y h(x,y),z+(y-x)_\kappa\> \big(\I_{\{|z+(y-x)_\kappa|\leq 1\}} -\I_{\{|z|\leq 1\}} \big) \mu_{(x-y)_\kappa}(dz)\\
  &\hskip14pt +\<b(x), \nabla_x h(x,y)\>+ \<b(y), \nabla_y h(x,y)\>,
  \end{align*}
where the first three integrals come from the integral in \eqref{coup-SDE-2} with respect to the Poisson random measure $\bar{N}(dt,dz,du)$, while the next two terms follow from the second integral in \eqref{coup-SDE-2}. Simplifying the above identity, we can easily see that $\bar L_X h(x,y)= \widetilde L_X h(x,y)$, therefore the proof is complete.
\end{proof}

According to the above discussions, $\widetilde L_X$ is a coupling operator of $L_X$ in \eqref{SDE-generator}, thus we deduce

\begin{corollary}
The process $(Y_t)_{t\geq 0}$ has the same finite dimensional distributions with $(X_t)_{t\geq 0}$.
\end{corollary}

Summarizing all the conclusions above, the coupling operator $\widetilde L_X$ generates a non-explosive coupling process $(X_t,Y_t)_{t\ge0}$ of the process $(X_t)_{t\ge0}$, and $X_t=Y_t$ for any $t\ge T,$ where $T=\inf\{t\ge0: X_t=Y_t\}$ is the coupling time of the process $(X_t,Y_t)_{t\ge0}$.

\begin{remark}\label{r-coulevy} Since the drift term $b$ can be chosen to be $b(x)=0$ for all $x\in \R^d$ in the proofs of Propositions \ref{2-prop-1} and \ref{3-prop-2}, one can claim that the process $(Z_t, Z^\ast_t)_{t\ge0}$ constructed in Subsection \ref{cou:pro} is a Markov coupling process for the L\'evy process $Z$, and its infinitesimal generator is $\widetilde L_Z$ defined in \eqref{proofth24}. In particular, the process $(Z^\ast_t)_{t\ge0}$ defined by \eqref{e:couppro} is also a L\'evy process on $\R^d$ with L\'evy measure $\nu$.  \end{remark}

\section{Exponential convergence in Wasserstein-type distances via coupling}\label{genecop}

By making full use of the coupling operator and the coupling process
constructed in Subsection \ref{sectioncousde}, we will provide in
this part a general result for exponential convergence in Wasserstein distances including the total variation.

\subsection{Preliminary calculations}\label{sec-pc}
Let $\widetilde{L}_X $ be the coupling operator given in \eqref{SDE-coup-op}. We will compute the expression of $\widetilde{L}_X f(|x-y|)$ for any $f\in C_b^1([0,\infty))$ with $f\ge0$.

Let $(X_t,Y_t)_{t\ge0}$ be the coupling process corresponding to the operator $\widetilde{L}_X $ constructed in Subsection \ref{sectioncousde}.
Recall that for any $t\ge0$, $\kappa\in (0,\kappa_0]$ and $z,u\in\R^d$,  $U_t=X_t-Y_t$ and
  $$V_{t}(z,u)=(U_{t})_{\kappa} \Big[ \I_{\{u\le \frac12 \rho((-U_{t})_{\kappa},z)\}} - \I_{\{\frac12 \rho((-U_{t})_{\kappa},z)< u\le \frac12 [\rho((-U_{t})_{\kappa},z)+\rho((U_{t})_{\kappa},z)] \}}\Big].$$ In particular,
  $$V_0(z,u)=(x-y)_{\kappa} \Big[ \I_{\{u\le \frac12 \rho((y-x)_{\kappa},z)\}} - \I_{\{\frac12 \rho((y-x)_{\kappa},z)< u\le \frac12 [\rho((y-x)_{\kappa},z)+\rho((x-y)_{\kappa},z)] \}}\Big].$$
It follows from the system \eqref{SDE-coup-eq-1} that
  \begin{equation*}\label{difference}
  \begin{split}
  dU_t &=(b(X_t)-b(Y_t))\, dt - \int_{\R^d\times [0,1]} V_{t-}(z,u)\, {N}(dt,dz,du).
  \end{split}
  \end{equation*}

Take $f\in C_b^1([0,\infty))$ with $f\ge0$. By the It\^{o} formula,
  \begin{align*}
  f(|U_t|)&= f(|x-y|)+ \int_0^t \frac{f'(|U_s|)}{|U_s|} \<U_s, b(X_s)-b(Y_s)\> \, ds\\
  &\quad +\int_0^t\int_{\R^d\times [0,1]} \big[f(|U_{s-}- V_{s-}(z,u)|)-f(|U_{s-}|)\big]\, N(ds,dz,du).
  \end{align*}
Therefore,
  \begin{align*}
  \widetilde L_X f(|x-y|) &= \frac{f'(|x-y|)}{|x-y|}\< b(x)-b(y),x-y\>\\
  &\quad + \int_{\R^d\times [0,1]} \big[f(|(x-y)-V_0(z,u)|) -f(|x-y|)\big]\, \nu(dz) \, du.
  \end{align*}
By the definition of $V_0$, the second term on the right hand side is equal to
  \begin{align*}
  &\hskip14pt \frac12 \int_{\R^d}\bigg[\big(f(|(x-y)-(x-y)_\kappa|) -f(|x-y|)\big)\rho((y-x)_{\kappa},z) \\
  &\hskip50pt + \big(f(|(x-y)+(x-y)_\kappa|) -f(|x-y|)\big)\rho((x-y)_{\kappa},z) \bigg] \nu(dz)\\
  &= \frac12 \Big[\big(f(|(x-y)-(x-y)_\kappa|) -f(|x-y|)\big)\mu_{(y-x)_{\kappa}}(\R^d)\\
  &\hskip50pt +\big(f(|(x-y)+(x-y)_\kappa|) -f(|x-y|)\big)\mu_{(x-y)_{\kappa}}(\R^d)\Big].
  \end{align*}
Thanks to the fact (also see Corollary \ref{C:mea}) that
  \begin{equation*}\label{eeff}
  \mu_{(y-x)_{\kappa}}(\R^d)= \mu_{(-(x-y))_{\kappa}}(\R^d) = \mu_{(x-y)_{\kappa}}(\R^d),
  \end{equation*}
we can finally conclude that, for any $x$, $y\in\R^d$ with $x\neq y$,
  \begin{equation}\label{proofth2544}\allowdisplaybreaks\aligned
  \widetilde{L}_X f(|x-y|)
    &=\frac{1}{2} \mu_{(x-y)_{\kappa}}(\R^d)\Big[f\big(|x-y|+\kappa\wedge |x-y|\big) +  f\big(|x-y|-\kappa\wedge |x-y|\big)\\
    &\hskip85pt  - 2f(|x-y|)\Big] +\frac{f'(|x-y|)}{|x-y|}\langle b(x)-b(y),x-y\rangle.
  \endaligned
  \end{equation}
Note that, by \eqref{proofth2544}, $\widetilde{L}_X f(|x-y|)$ is pointwise well defined for any $f\in C^1([0,\infty))$.

\subsection{General result}

The following theorem provides us a general result for exponential convergence in Wasserstein-type distance via the coupling method. Recall the definition of $J(s)$ at the beginning of Section 1.1.

\begin{theorem}\label{p-w}
Assume that the drift term $b$ satisfies ${\mathbf B(\Phi_1, \Phi_2, l_0)}$, i.e. \eqref{th111}, and that \eqref{th10} holds for the L\'{e}vy measure $\nu$ with some $\kappa_0>0$. For any $n\ge 1$, let $\psi_n\in C^1([0,\infty))$ be increasing on $[0,\infty)$, satisfying $\psi_n(0)=0$ and
  \begin{equation}\label{p-w-1}
  \psi_n(r+s)+\psi_n(r-s)-2\psi_n(r)\leq 0\quad \mbox{for all } r\geq 1/n,\, 0<s\le r\wedge \kappa_0.
\end{equation}
Suppose that there are $\lambda>0$ and $\kappa\in (0,\kappa_0]$ such that for $n\ge l_0^{-1}\vee l_0$ large enough, $\psi_n$ satisfies the condition ${\mathbf C(\lambda, \kappa, n)}$ on $[1/n,n]$ as follows:
\begin{itemize}

\item[\rm(i)] for $r\in[1/n,l_0)$,
  $$\frac{1}{2}J(\kappa\wedge r) \big[\psi_n(r+r\wedge \kappa)+\psi_n(r-r\wedge \kappa)-2\psi_n(r) \big] +\Phi_1(r)\psi_n'(r)\le -\lambda \psi_n(r);$$

\item[\rm(ii)] for $ r\in[l_0,n]$,
  $$-\Phi_2(r)\psi_n'(r)\le -\lambda \psi_n(r).$$

\end{itemize}
Then for any $t>0$ and $x,y\in\R^d$,
  \begin{equation}\label{ep-w}
  W_{\psi_\infty}(\delta_xP_t, \delta_y P_t)\leq \psi_{\infty}(|x-y|)e^{-\lambda t},
  \end{equation}
where  $\psi_\infty=\liminf_{n\to\infty}\psi_n.$
\end{theorem}

In applications, the limit function $\psi_{\infty}$ is finite on $[0,\infty)$, hence \eqref{ep-w} implies the finiteness of $W_{\psi_\infty}(\delta_xP_t, \delta_y P_t)$ for all $x,y\in \R^d$ and $t>0$.

\begin{proof}[Proof of Theorem $\ref{p-w}$]
\emph{Step $1$.} Let $\widetilde{L}=\widetilde{L}_X $ be the coupling operator given in \eqref{SDE-coup-op}. We first prove that for $n\ge l_0^{-1}\vee l_0$ large enough and for all $x,y\in \R^d$ with $1/n\le |x-y|\le n$,
  \begin{equation}\label{esope}
  \widetilde{L}\psi_n(|x-y|)\le - \lambda \psi_n(|x-y|).
  \end{equation}
For this, we consider the following two cases.
\begin{itemize}
\item[(a)] $1/n\le |x-y|< l_0$. The definition of $J(s)$ leads to
  $$\mu_{(x-y)_{\kappa}}(\R^d)= \big[{\nu}\wedge(\delta_{\left(1\wedge \frac{\kappa}{|x-y|}\right)(x-y)}\ast {\nu})\big](\R^d)\geq J(|x-y|\wedge \kappa).$$
Thus by \eqref{proofth2544}, \eqref{p-w-1} and \eqref{th111},
  \begin{align*}
  \widetilde{L} \psi_n(|x-y|)&\leq \frac12 J(|x-y|\wedge \kappa)\Big[\psi_n(|x-y|+|x-y|\wedge \kappa)\\
  &\hskip40pt +\psi_n(|x-y|-|x-y|\wedge \kappa)-2\psi_n(|x-y|) \Big]\\
  &\hskip13pt + \Phi_1(|x-y|)\psi_n'(|x-y|)\\
  &\leq -\lambda \psi_n(|x-y|),
  \end{align*}
where we used the condition (i) in the last inequality.

\item[(b)] $l_0\le |x-y|\le n$. In view of \eqref{proofth2544}, it is obvious from the conditions \eqref{p-w-1} and \eqref{th111} that
  $$\widetilde{L}\psi_n(|x-y|)\le-\Phi_2(|x-y|)\psi'_n(|x-y|)\le - \lambda \psi_n(|x-y|),$$
where the last inequality follows from (ii).
\end{itemize}
Then \eqref{esope} is proved by summarizing these arguments.

\emph{Step $2$.} Based on \eqref{esope}, the proof of the desired
assertion \eqref{ep-w} is similar to that of \cite[Theorem 1.3]{LW}
or \cite[Theorem 1.2]{Wang15} by some slight modifications. For the
sake of completeness, we present the details here. Let
$(X_t,Y_t)_{t\ge0}$ be the coupling process constructed in
Subsection \ref{sectioncousde}. It suffices to verify that for
$x,y\in\R^d$ with $|x-y|>0$ and any $t>0$,
  $$ \widetilde{\Ee}^{(x,y)}\psi_\infty(|X_t-Y_t|)\leq \psi_\infty(|x-y|)e^{-\lambda t},$$
where $\widetilde{\Ee}^{(x,y)}$ is the expectation of $(X_t,Y_t)_{t\ge0}$ starting from $(x,y)$.

For any $t>0$ set $r_t=|U_t|=|X_t-Y_t|$, and for $n\geq 1$ define the stopping time
  $$T_n=\inf\{t>0: r_t\notin [1/n, n]\}.$$
Since the coupling process $(X_t,Y_t)_{t\ge0}$ is non-explosive, we have $T_n\uparrow T$ a.s. as $n\to\infty$, where $T$ is the coupling time of the process $(X_t,Y_t)_{t\ge0}$.

For any $x,$ $y\in\R^d$ with $|x-y|>0$, we take $n\ge l_0^{-1}\vee l_0$ large enough such that $1/n<|x-y|<n$.  For $m\ge n$, let $\psi_m$ be the function and $\lambda$ be the constant given in the statement. Then,
  $$\aligned
  &\widetilde{\Ee}^{(x,y)}\big[e^{\lambda(t\wedge T_{n})} \psi_m(|X_{t\wedge T_{n}}-Y_{t\wedge T_{n}}|)\big]\\
  &=\psi_m(|x-y|)+\widetilde{\Ee}^{(x,y)}\bigg(\int_0^{t\wedge T_{n}} e^{\lambda s}\big[\lambda\psi_m(|X_{s}-Y_{s}|)+\widetilde{L} \psi_m(|X_{s}-Y_{s}|)\big]\,d s\bigg)\\
  &\le \psi_m(|x-y|),\endaligned$$ where the inequality above follows from \eqref{esope}.
Hence,
  $$\widetilde{\Ee}^{(x,y)} \big[e^{\lambda (t\wedge T_{n})}\psi_m(r_{t\wedge T_n})\big]\leq \psi_m(r_0).$$
 Thus by Fatou's lemma, first letting $m\to \infty$ and then $n\to\infty$ in the above inequality gives us
  $$ \widetilde{\Ee}^{(x,y)}\big(e^{\lambda(t\wedge T)}\psi_\infty(r_{t\wedge T})\big)\leq \psi_\infty(r_0). $$
Thanks to our convention that $Y_t=X_t$ for $t\geq T$, we have $r_t=0$ and so $\psi_\infty(r_t)=0$ for all $t\geq T$, which implies
  $$ \widetilde{\Ee}^{(x,y)}\big(e^{\lambda(t\wedge T)}\psi_\infty(r_{t\wedge T})\big)=e^{\lambda t} \widetilde{\Ee}^{(x,y)} \big(\psi_\infty(r_t)\I_{\{T>t\}}\big) = e^{\lambda t}\widetilde{\Ee}^{(x,y)} \psi_\infty(r_t).$$
Therefore, the desired assertion \eqref{ep-w} follows from all the discussions above.
\end{proof}

\section{General results and proofs}\label{section3}

\subsection{Proofs of results related to Wasserstein-type distances}

The following result is crucial for constructing test functions $\psi_n$ in Theorem \ref{p-w}.

\begin{lemma}\label{lem-test-funct}
Let $g\in C([0,2l_0])\cap C^3((0,2l_0])$ be satisfying $g(0)=0$ and
  \begin{equation}\label{lem-test-funct.0}
  g'(r)\geq 0,\ g''(r)\leq 0 \mbox{ and } g'''(r)\geq 0\quad \mbox{for any } r\in (0,2l_0].
  \end{equation}
Then for all $c_1, c_2>0$ the function
  \begin{equation}\label{lem-test-funct.1}
  \psi(r):=\psi_{c_1,c_2}(r)= \begin{cases}
  c_1 r+ \int_0^r e^{-c_2 g(s)}\, ds ,& r\in [0, 2l_0],\\
 \psi(2l_0)+ \psi'(2l_0) (r-2l_0), & r\in(2l_0,\infty)
  \end{cases}
  \end{equation}
satisfies
\begin{itemize}
\item[\rm (1)] $\psi\in C^1([0,\infty))$ and $c_1r\le \psi(r)\le (c_1+1)r$ on $[0,2l_0]$;
\item[\rm (2)]$\psi'> 0$, $\psi''\leq 0,$  $\psi'''\geq 0$ and $\psi^{(4)}\leq 0$ on $(0,2l_0]$;
\item[\rm (3)] for any $0\le \delta\le r$, $$ \psi(r+\delta)+\psi(r-\delta)-2\psi(r)\leq 0;$$
\item[\rm (4)] for any  $0\leq \delta \leq r\leq l_0$,
  $$\psi(r+\delta) +\psi(r-\delta)-2\psi(r)\leq \psi''(r)\delta^2.$$
\end{itemize}
\end{lemma}

\begin{proof} (1) is trivial. The property (2) follows from \eqref{lem-test-funct.0} and the definition of $\psi$ by direct calculations. The assertion (3) is trivial if $\delta=0$, thus we assume $\delta>0$ in the sequel. By the mean value formula, there exist constants $\xi_1\in (r,r+\delta)$ and $\xi_2\in (r-\delta, r)$ such that
  $$\psi(r+\delta)-\psi(r)=\psi'(\xi_1)\delta$$
and
  $$\psi(r-\delta)-\psi(r)= -\psi'(\xi_2)\delta.$$
Therefore,
  $$\psi(r+\delta)+\psi(r-\delta)-2\psi(r)=(\psi'(\xi_1)-\psi'(\xi_2))\delta\le 0,$$
since $\psi'$ is decreasing due to the definition of $\psi$.

To prove (4), we will still assume $\delta>0$. Similar to the proof of (3), by the Taylor formula, there exist constants $\xi_1\in (r,r+\delta)$ and $\xi_2\in (r-\delta, r)$ such that
  \begin{align*}
  \psi(r+\delta)&=\psi(r)+ \psi'(r)\delta +\frac12 \psi''(r)\delta^2 +\frac16 \psi'''(\xi_1)\delta^3,\\
  \psi(r-\delta)&=\psi(r)- \psi'(r)\delta +\frac12 \psi''(r)\delta^2 -\frac16 \psi'''(\xi_2)\delta^3.
  \end{align*}
Therefore,
  $$\psi(r+\delta) +\psi(r-\delta)-2\psi(r)= \psi''(r)\delta^2 +\frac{\delta^3}6 \big[\psi'''(\xi_1)- \psi'''(\xi_2)\big] \leq \psi''(r)\delta^2$$
since $\psi'''$ is decreasing due to (2).
\end{proof}

In the next theorem we establish the exponential contraction in $L^1$-Wasserstein distance which is more general than assertion (a) in Theorem \ref{th1}.

\begin{theorem}\label{thtpw}
Assume that
\begin{itemize}
\item[\rm (a)] \eqref{th10} holds for the L\'{e}vy measure $\nu$ with some $\kappa_0>0$;
\item[\rm (b)] the drift $b$ satisfies ${\mathbf B(\Phi_1(r), K_2 r, l_0)}$ for some constants $ K_2>0,\,l_0\ge 0$, and a nonnegative concave function $\Phi_1\in C([0,2l_0])\cap C^2((0,2l_0])$ such that $\Phi_1(0)=0$ and $\Phi''_1$ is nondecreasing;
\item[\rm (c)]  there is a nondecreasing and concave function $\sigma\in C([0,2l_0])\cap C^2((0,2l_0])$ such that for some $\kappa\in (0,\kappa_0]$, one has
  \begin{equation}\label{thtpw.1}
  \sigma(r)\leq \frac1{2r} J(\kappa\wedge r) (\kappa\wedge r)^2, \quad r\in (0, 2l_0];
  \end{equation}
and the integrals $g_1(r)=\int_{0}^r \frac{1}{\sigma(s)}\,ds$ and $g_2(r)=\int_0^r \frac{\Phi_1(s)}{s\sigma(s)}\, ds$ are well defined for all $r\in [0,2l_0]$.
\end{itemize}
Set $c_2=(2K_2)\wedge g_1(2l_0)^{-1}$ and $c_1=e^{-c_2 g(2l_0)}$, where the function $g$ is defined by
  $$g(r)=g_1(r)+\frac2{c_2} g_2(r),\quad r\in (0,2l_0].$$ Let $\psi$ be defined by \eqref{lem-test-funct.1} with $c_1,c_2$ and $g$ given above.
Then for any $x,y\in\R^d$ and $t>0$,
  $$  W_\psi(\delta_xP_t, \delta_yP_t)\le e^{-\lambda t}  \psi(|x-y|)$$
and
  $$  W_1(\delta_xP_t, \delta_yP_t)\le C e^{-\lambda t} |x-y|,$$
where
  \begin{equation}\label{pwcon}\begin{split}
   C& =\frac{1+c_1}{2c_1}=\frac1{2}\Big(1+\exp\big\{g(2l_0) \big[ (2K_2)\wedge g_1(2l_0)^{-1} \big] \big\}\Big),\\
    \lambda&=\frac{c_2}{1+e^{c_2 g(2l_0)}}= \frac{(2K_2)\wedge g_1(2l_0)^{-1}}{1+ \exp\big\{ g(2l_0) \big[ (2K_2)\wedge g_1(2l_0)^{-1} \big] \big\}}.
  \end{split}
  \end{equation}

\end{theorem}

Before going to the proof, we make some comments.

\begin{remark}\label{r:thtpw}\rm
(1) When $l_0=0$, the drift term $b$ satisfies the uniformly dissipative condition, i.e. for any $x,y\in\R^d$,
 $$\langle b(x)-b(y),x-y\rangle\le -K_2|x-y|^2.$$
By using the classical synchronous coupling, one can prove that for any $x,y\in\R^d$ and $t>0$,
  $$  W_1(\delta_xP_t, \delta_yP_t)\le e^{-K_2 t}|x-y|.$$
In this case, the constants $C$ and $\lambda$ given by \eqref{pwcon} are also equal to $1$ and $K_2$, respectively.

(2) Assume that $\Phi_1(s)=K_1 s$ for any $s\in [0, 2l_0]$ with some constant $K_1\ge 0$. Then
  $$g(r)=\int_0^r \frac{ds}{\sigma(s)}+\frac{2}{c_2}\int_0^r \frac{K_1 s}{s\sigma(s)}\, ds =\frac{c_2+2K_1 }{c_2} g_1(r),\quad r\in (0,2l_0],$$
and so
  $$c_2g(2l_0)={(2K_1+c_2)g_1(2l_0)}={\big[2K_1+ (2K_2)\wedge g_1(2l_0)^{-1}\big] g_1(2l_0)}\le 2K_1g_1(2l_0)+1.$$
Therefore, for any fixed $a_0>0$, we have
  $$\lambda \ge  \begin{cases}
  \big(1+{e^{1+2a_0}}\big)^{-1} \big[(2K_2)\wedge g_1(2l_0)^{-1}\big], &  K_1 g_1(2l_0) \le a_0;\\
  (2e)^{-1} \big[(2K_2)\wedge g_1(2l_0)^{-1}\big] \exp\big(-2K_1 g_1(2l_0)\big), &  K_1 g_1(2l_0)\ge a_0.
  \end{cases}$$
According to Example \ref{est-stable}, when $Z$ is the (truncated)
symmetric $\alpha$-stable process with $\alpha\in (0,2)$, the lower
bounds above for $\lambda$ are of optimal orders with respect to
$l_0, K_1$ and $K_2$ when $\alpha\to 2$ (i.e.\ $Z$ is replaced by
the standard Brownian motion).

(3) Suppose that  \eqref{th10} holds with some $\kappa_0>0$ and
  $$\lim_{\kappa\to0} J(\kappa) \kappa^2 =0,$$
which are true for (truncated) $\alpha$-stable processes with $\alpha\in (0,2)$, cf. the proof of Example \ref{exa}. We claim that, if $l_0>0$, then the constant $\lambda$ defined in \eqref{pwcon} tends to $0$ as $\kappa \to0$. Indeed, for $\kappa< r\leq 2l_0$, one has
  $$\sigma(r)\leq \frac1{2r} J(\kappa\wedge r) (\kappa\wedge r)^2 = \frac1{2r} J(\kappa) \kappa^2,$$
hence, as $\kappa\to0$,
  $$g_1(2l_0)=\int_0^{2l_0} \frac{dr}{\sigma(r)}\geq \int_\kappa^{2l_0} \frac{2r}{J(\kappa) \kappa^2}\, dr=\frac{4l_0^2-\kappa^2}{J(\kappa) \kappa^2} \to\infty$$ and
so
$$\lambda\le (2K_2)\wedge g_1(2l_0)^{-1}\to 0.$$
\end{remark}

Next, we are in a position to present the

\begin{proof}[Proof of Theorem $\ref{thtpw}$]
We split the proof into two steps.

\emph{Step $1$.} We first show that the function $g$ defined in the theorem satisfies \eqref{lem-test-funct}. For $r\in (0,2r_0]$, it is clear that
  $$g'(r)=\frac1{\sigma(r)}\bigg[1 + \frac{2\Phi_1(r)}{c_2 r}\bigg]\geq 0.$$
Next, since $\Phi_1$ is concave and $\Phi_1(0)=0$, we have $\Phi_1(r)=\int_0^r \Phi_1'(s)\, ds\geq \Phi_1'(r) r$. This together with $\sigma'\geq 0$ implies
  $$g''(r)=- \frac{\sigma'(r)}{\sigma(r)^2}\bigg[1 + \frac{2\Phi_1(r)}{c_2 r}\bigg] +\frac2{c_2\sigma(r)}\cdot \frac{\Phi_1'(r)r -\Phi_1(r)}{r^2}\leq 0.$$
Finally,
  $$\aligned
  g'''(r)&=\frac{2\sigma'(r)^2-\sigma(r)\sigma''(r)}{\sigma(r)^3} \bigg[1 + \frac{2\Phi_1(r)}{c_2 r}\bigg] -\frac{4\sigma'(r)}{c_2\sigma(r)^2} \cdot \frac{\Phi_1'(r)r -\Phi_1(r)}{r^2}\\
  &\hskip14pt  + \frac{2}{c_2\sigma(r)}\cdot \frac{2\Phi_1(r)-2\Phi_1'(r)r+ \Phi_1''(r) r^2}{r^3}. \endaligned$$
As $\sigma''(r)\leq 0$, the first term on the right hand side is nonnegative. The same is true for the second term since $\sigma'(r)\geq 0$ and $\Phi_1'(r)r -\Phi_1(r)\leq 0$. For the last term, we have by Taylor's formula that there is a constant $\xi\in (0,r)$ such that
  $$\Phi_1(0)=\Phi_1(r)-\Phi_1'(r)r +\frac12 \Phi_1''(\xi) r^2 \leq \Phi_1(r)-\Phi_1'(r)r +\frac12 \Phi_1''(r) r^2,$$
where the last inequality is due to the fact that $\Phi_1''$ is nondecreasing. Note that $\Phi_1(0)=0$, we conclude that the third term is also nonnegative. Therefore $g'''(r)\geq 0$.

\emph{Step $2$.} Let $\psi$ be defined by \eqref{lem-test-funct.1} with $c_1,c_2$ and $g$ given in the theorem. We prove that $\psi$ satisfies ${\mathbf C(\lambda, \kappa, \infty)}$ for some $\lambda>0$ and $\kappa\in(0,\kappa_0]$ (see Theorem \ref{p-w} for its meaning). Note that, by (3) in Lemma \ref{lem-test-funct}, $\psi$ verifies \eqref{p-w-1} for all $r\geq s\ge 0$. Under the condition ${\mathbf B(\Phi_1(r), K_2 r, l_0)}$, (4) in Lemma \ref{lem-test-funct} and \eqref{thtpw.1} yield that for all $r\in (0,l_0]$,
  \begin{eqnarray*}
  \Theta(r)&:=& \frac12 J(\kappa\wedge r) \big[\psi(r+\kappa\wedge r)+\psi(r-\kappa\wedge r)-2\psi(r) \big] + \Phi_1(r)\psi'(r)\\
  &\leq& \frac12 J(\kappa\wedge r) (\kappa\wedge r)^2 \psi''(r)+\Phi_1(r)\psi'(r)\leq  \sigma(r) r\psi''(r)+\Phi_1(r)\psi'(r).
  \end{eqnarray*}
By \eqref{lem-test-funct.1}, we have $\psi'(r)= c_1+ e^{-c_2 g(r)}$ and $\psi''(r)= -c_2g'(r) e^{-c_2 g(r)}$. Hence, by the definition of $g$, we get that
  \begin{equation}\label{rr1}\begin{split}
  \Theta(r)&\leq \sigma(r) r\big[-c_2g'(r) e^{-c_2 g(r)}\big]+ \Phi_1(r)\big[c_1+ e^{-c_2g(r)}\big]\\
  &\leq -c_2 r e^{-c_2 g(r)}\bigg[1+\frac{2\Phi_1(r)}{c_2 r}\bigg] + 2\Phi_1(r) e^{-c_2 g(r)}\\
  &=  -c_2r e^{-c_2 g(r)} \leq -c_1 c_2 r\leq -\frac{ c_1 c_2}{c_1+1} \psi(r),
  \end{split}\end{equation}
where the last inequality follows from (1) in Lemma \ref{lem-test-funct}.

Next, if $r\in (l_0, 2l_0]$, by ${\mathbf B(\Phi_1(r), K_2 r, l_0)}$ and (1) in Lemma \ref{lem-test-funct} again,
  \begin{equation}\label{rr2}\begin{split}
  -K_2 r\psi'(r) &= -K_2 r \big[c_1+ e^{-c_2 g(r)}\big] \leq -\frac{K_2 [c_1+ e^{-c_2g(2l_0)}]}{c_1+1} \psi(r)\\
  &= -\frac{2K_2 c_1}{c_1+ 1} \psi(r)\leq -\frac{ c_1 c_2}{c_1+ 1} \psi(r).
  \end{split}\end{equation}
Note that the function
  $$r\mapsto \frac{\psi'(2l_0)\,r}{\psi(r)}=\frac{2c_1 r}{2c_1 r+\int_0^{2l_0} e^{-c_2 g(s)}\,d s- 2c_1 l_0}$$
is increasing on $(2l_0,\infty)$, since $\int_0^{2l_0} e^{-c_2 g(s)}\,d s\ge 2l_0 e^{-c_2 g(2l_0)}= 2 c_1 l_0.$ Thus for $r>2l_0$, we use again ${\mathbf B(\Phi_1(r), K_2 r, l_0)}$ to obtain
  \begin{equation}\label{rr3}
  \begin{split}
  -K_2 r\psi'(r) & =-K_2\psi'(2l_0)\,r\le  -K_2\frac{2l_0\psi'(2l_0)}{\psi(2l_0)}\psi(r)\\
  &\le -2 K_2\frac{2c_1 l_0}{2l_0(c_1+1)} \psi(r)\leq -\frac{c_1 c_2 }{c_1+1} \psi(r).
 \end{split}\end{equation}

We conclude from all the estimates above that ${\mathbf C(\lambda, \kappa, \infty)}$ holds with the positive constant $\lambda$ given by \eqref{pwcon}. Therefore, we can apply Theorem \ref{p-w} to get that for any $t>0$ and $x,y\in\R^d$,
  $$W_\psi(\delta_xP_t, \delta_y P_t)\leq \psi(|x-y|)e^{-\lambda t}. $$
Since $\psi$ is concave on $[0,\infty)$, it is clear that $(c_1+1)r\geq \psi(r)\geq \psi'(2l_0)r=2c_1 r$ for all $r\geq 0$. Hence the desired result holds with $C=(c_1+1)/(2c_1)$.
\end{proof}

Similar to Theorem \ref{thtpw}, we have the following statement about the exponential rates for total variation.

\begin{theorem}\label{thtvart}
Assume that the drift $b$ satisfies ${\mathbf B(K_1, K_2 r, l_0)}$ for some $K_1, l_0\ge 0$ and $K_2>0$, and that \eqref{th10} holds for the L\'{e}vy measure $\nu$ with some $\kappa_0>0$. Moreover, suppose that there is a nondecreasing and concave function $\sigma\in C([0,2l_0])\cap C^2((0,2l_0])$ such that for some $\kappa\in (0,\kappa_0\wedge l_0]$, one has
  $$ \sigma(r)\leq \frac1{2r} J(\kappa\wedge r) (\kappa\wedge r)^2, \quad r\in (0, 2l_0]; $$
and the function $g(r)=\int_{0}^r \frac{ds}{\sigma(s)}$ is well defined for all $r\in [0,2l_0]$. Then there exist constants $\lambda, c>0$ such that for any $x,y\in\R^d$ and $t>0$,
  $$ W_1(\delta_xP_t, \delta_yP_t)+ \|\delta_xP_t- \delta_yP_t\|_{{\rm Var}}\le c e^{-\lambda t}(1+|x-y|).$$
\end{theorem}

\begin{proof}\emph{Step $1$.} Let $\psi$ be the function defined by \eqref{lem-test-funct.1}. For any $n \ge1$, define $\psi_{n}\in C^2([0,\infty))$ such that $\psi_{n}$ is strictly increasing and
   $$\psi_{n}(r)\begin{cases}
      = \psi(r),        &  0\le r\le 1/(n+1);\\
      \le a+\psi(r),    & 1/(n+1)< r\le 1/n;\\
      = a+\psi(r),      &1/n\le r< \infty,
    \end{cases}$$
where $a>0$ and the constants $c_1, c_2$ in the definition of $\psi$ are determined later. For any $n\ge1$ and every $r\in [1/n,\infty)$, we have $\psi_{n}(r)=a+\psi(r)$ and $\psi'_{n}(r)=\psi'(r)$. Therefore, for any $\kappa\in (0,\kappa_0]$,
  \begin{equation}\label{p-var1}
  \psi_n(r-\kappa\wedge r) =\psi_n(r-\kappa \wedge r)\I_{\{r>\kappa\}}\leq \big[a+\psi(r-\kappa \wedge r)\big]\I_{\{r>\kappa\}}.
  \end{equation}
This along with (3) in Lemma \ref{lem-test-funct} implies that $\psi_n$ fulfills \eqref{p-w-1}.

Below we prove that by proper choices of $c_1$, $c_2$ and $a>0$, for $n\ge l_0^{-1}\vee l_0$ large enough, $\psi_{n}$ satisfies ${\mathbf C(\lambda, \kappa, n)}$ with some constants $\lambda>0$ and $\kappa\in (0,\kappa_0]$  (indeed for all $r\in [1/n,\infty)$). Once this is done, then, by Theorem \ref{p-w} and the fact that
  $$\lim_{n\to \infty} \psi_n=a\I_{(0,\infty)}+\psi,$$
we have for any $x,y\in\R^d$ with $x\neq y$,
  $$W_{a\I_{(0,\infty)}+\psi}(\delta_xP_t, \delta_y P_t)\le e^{-\lambda t}(a+\psi(|x-y|)). $$
This implies that \begin{align*}W_1(\delta_xP_t, \delta_yP_t)+\|
\delta_xP_t- \delta_y P_t\|_{\var}&\le \Big(\frac{1}{2c_1} \vee
\frac{2}{a}\Big)
W_{a\I_{(0,\infty)}+\psi}(\delta_xP_t, \delta_y P_t)\\
&\le \Big(\frac{1}{2c_1} \vee \frac{2}{a}\Big) e^{-\lambda
t}\Big(a+\psi(|x-y|)\Big)\\
&\le \Big(\frac{1}{2c_1} \vee \frac{2}{a}\Big) \Big((1+c_1)\vee
a\Big) e^{-\lambda t}\big(1+|x-y|\big),\end{align*} which proves the
desired assertion.

\emph{Step $2$.} In the proof below we also aim to give an explicit expression for the exponential rate $\lambda$ in the theorem. First, by \eqref{th10}, for any $0<\kappa\le \kappa_0$,
  $$J_{\kappa}:=\inf_{0<s\le \kappa} J(s)>0.$$
Note that the drift term $b$ satisfies ${\mathbf B(K_1,K_2r, l_0)}$ for some $K_1, l_0\ge0$ and $K_2>0$. According to \eqref{p-var1}, for $r\in (1/n,l_0]$, we have
  \begin{equation}\label{proof-var-1}
  \begin{split}
  \Theta_n(r) &:= \frac{1}{2}J(\kappa \wedge r) \big[\psi_n(r+r\wedge \kappa)+\psi_n(r-r\wedge \kappa)-2\psi_n(r) \big] +K_1\psi_n'(r)\\
  &\, \le \frac{1}{2}J(\kappa \wedge r) \big[\psi(r+r\wedge \kappa)+\psi(r-r\wedge \kappa)-2\psi(r) \big] +K_1\psi'(r)\\
  &\, \quad- \frac{a}{2}J(\kappa \wedge r)\I_{\{r\le \kappa \wedge l_0\}}.
  \end{split}
  \end{equation}

In the following, let $\kappa\in (0,\kappa_0\wedge l_0]$ be the constant in assumptions of the theorem. By \eqref{proof-var-1} and (4) in Lemma \ref{lem-test-funct}, we find that for all $r\in(\kappa, l_0]$,
  \begin{align*}
  \Theta_n(r) &\le \frac{1}{2} J(\kappa)\kappa^2\psi''(r)+K_1\psi'(r)\le \frac{1}{2} J(\kappa)\kappa^2\psi''(r)+\frac{K_1}{\kappa}r\psi'(r).
  \end{align*}
Taking $c_1= e^{-c_2g(2l_0)}$ and $c_2=2K_1/\kappa+ \big[(2K_2)\wedge g(2l_0)^{-1}\big] $, and following the argument of \eqref{rr1},
we obtain that for all $r\in(\kappa, l_0]$,
\begin{equation}\label{sim-prof1}
  \Theta_n(r) \le -\frac{ c_1}{c_1+ 1} \big[(2K_2)\wedge g(2l_0)^{-1}\big]\psi(r).
\end{equation}
On the other hand, we can deduce from \eqref{proof-var-1} and (3) in Lemma \ref{lem-test-funct} that for all $r\in[1/n,\kappa]$,
  \begin{align*}
  \Theta_n(r) &\le K_1(c_1+ e^{-c_2g(r)})-\frac{a}{2}J_{\kappa} \le K_1(c_1+1)-\frac{a}{2}J_{\kappa}.
  \end{align*}
Then, choosing
  $$a=\frac{2}{J_{\kappa}}\left(K_1(c_1+ 1)+\frac{ c_1}{c_1+ 1} \big[(2K_2)\wedge g(2l_0)^{-1}\big]\psi(\kappa)\right),$$
we find that for all $r\in[1/n,\kappa]$,
  \begin{equation}\label{sim-prof2}
  \Theta_n(r) \le - \frac{c_1}{c_1+ 1} \big[(2K_2)\wedge g(2l_0)^{-1}\big] \psi(\kappa).
  \end{equation}
Furthermore, using ${\mathbf B(K_1,K_2r, l_0)}$ and following the arguments of \eqref{rr2} and \eqref{rr3}, it is easy to see that for all $r\ge l_0$,
  $$ -K_2r\psi_n'(r)=-K_2r\psi'(r) \le -\frac{2K_2 c_1}{c_1+ 1} \psi(r).$$

Combining all the estimates above, we can see that $\psi_{n}$ satisfies ${\mathbf C(\lambda, \kappa, n)}$ with
  \begin{equation}\label{convar}\begin{split}
  \lambda&=\frac{c_1}{c_1+ 1}\big[(2K_2)\wedge g(2l_0)^{-1}\big] \inf_{r>0} \frac{\psi(r\vee \kappa )}{a+\psi(r)}\\
  &=\frac{c_1}{c_1+ 1} \big[(2K_2)\wedge g(2l_0)^{-1}\big] \left(1+\frac{a}{\psi(\kappa)}\right)^{-1}>0.
  \end{split}\end{equation}
Then, the proof is complete.
\end{proof}

\begin{remark} Suppose that $$\lim_{\kappa\to0} \inf_{0<s\le \kappa} J(s)=\infty,$$ which holds true under \eqref{th13}.
If the drift term $b$ satisfies ${\mathbf B(K_1,K_2r, l_0)}$ with $K_1=0$, then for any $x,y\in\R^d$,
  $$\langle b(x)-b(y),x-y\rangle\le 0.$$
In this case, the exponential rate $\lambda$ given by \eqref{convar} is reduced into
  \begin{align*}\lambda= &\frac{(2K_2)\wedge{g(2l_0)}^{-1}}{1+ \exp\big\{g(2l_0)\big[(2K_2)\wedge {g(2l_0)}^{-1}\big]\big\}}\Bigg(1+\frac{2}{J_\kappa}\cdot\frac{(2K_2)\wedge {g(2l_0)}^{-1}}{1+ \exp\big\{g(2l_0)\big[(2K_2)\wedge {g(2l_0)}^{-1}\big]\big\}}\Bigg)^{-1}.
  \end{align*}
Note that, as $\kappa \to0$,
  $$\frac{2}{J_\kappa}\cdot\frac{(2K_2)\wedge {g(2l_0)}^{-1}}{1+ \exp\big\{g(2l_0)\big[(2K_2)\wedge {g(2l_0)}^{-1}\big]\big\}}\le \frac{4K_2}{J_\kappa} \to 0,$$
thanks to  $\lim_{\kappa\to0}J_\kappa=\infty$. Therefore, the quantity in the big round brackets tends to 1 as $\kappa \to0$, which implies that the exponential rate with respect to the total variation can be arbitrarily close to the one with respect to the $L^1$-Wasserstein distance (by choosing $\kappa$ small enough), provided that the condition ${\mathbf B(K_1,K_2r, l_0)}$ holds with $K_1=0$.
\end{remark}

We can now present the

\begin{proof}[Proof of  Theorem $\ref{th1}$]
Our strategy is to deduce the assertions (a) and (b) from Theorems \ref{thtpw} and \ref{thtvart}, respectively. It is obvious that $\Phi_1(r)=K_1 r^\beta$ with $\beta\in(0,1]$ satisfies (b) in Theorem \ref{thtpw}. Hence, it suffices to show that, under the condition \eqref{th13}, there exists a function $\sigma\in C([0,2l_0])\cap C^2((0,2l_0])$ satisfying the conditions in Theorems \ref{thtpw} and \ref{thtvart}.

Let $b_0=2l_0 e^{(1+\theta)/(1-\alpha)}$. It is easy to see that the function $s\mapsto s^{1-\alpha} \left(\log \frac{b_0}{s} \right)^{1+\theta}$ is concave and increasing on the interval $[0,2l_0]$. Under the condition \eqref{th13}, there exist constants $\kappa\in (0,\kappa_0\wedge l_0 \wedge1]$ and $b_1>0$ such that for all $s\in (0,\kappa]$, it holds
  $$J(s)\geq b_1 s^{-\alpha}\left(\log\frac{1}{s}\right)^{1+\theta}.$$
By taking a smaller $b_1$ we also have
  $$J(s)\geq b_1 s^{-\alpha} \left(\log\frac{b_0}{s}\right)^{1+\theta}, \quad s\in (0,\kappa],$$
which is equivalent to
  $$\frac12 s J(s)\geq \frac{b_1}2 s^{1-\alpha}\left(\log\frac{b_0}{s}\right)^{1+\theta},\quad s\in (0,\kappa].$$
For $s\in (\kappa,2l_0]$, we have
  $$\frac1{2s} J(\kappa) \kappa^2\geq \frac1{4l_0}J(\kappa) \kappa^2>0.$$
From the above two inequalities, we deduce that there is a small enough constant $b_2\in (0,b_1/2)$ such that
  $$b_2 s^{1-\alpha}\left(\log\frac{b_0}{s}\right)^{1+\theta}\leq \frac1{2s} J(\kappa\wedge s) (\kappa\wedge s)^2 \quad \mbox{for all } s\in (0,2l_0].$$
That is, \eqref{thtpw.1} holds with $\sigma(s)=b_2 s^{1-\alpha}\left(\log\frac{b_0}{s}\right)^{1+\theta}$. It is clear that the integrals $\int_0^r \frac{1}{\sigma(s)}\, ds$ and $\int_0^r \frac{K_1 s^\beta}{s\sigma(s)}\, ds$ are well defined since $\alpha\geq 0$ and $\alpha+\beta\geq 1$. Therefore the function $\sigma$ satisfies all the requirements in Theorem \ref{thtpw}. Note that $\int_0^r \frac{1}{\sigma(s)}\, ds$ still makes sense when $\alpha=0$, thus it fulfills also the conditions in Theorem \ref{thtvart}.
\end{proof}

To conclude this subsection, we present the proofs of Examples
\ref{exa} and \ref{est-stable}.

\begin{proof}[Proof of Example $\ref{exa}$]
Denote by $q(z)=\I_{\{0<z_1\leq 1\}} \frac{c_{d,\alpha}}{|z|^{d+\alpha}}$ for any $z\in\R^d$. Then
  $$q(z)\wedge q(x+z)= \bigg(\I_{\{0<z_1\leq 1\}} \frac{c_{d,\alpha}}{|z|^{d+\alpha}}\bigg) \wedge \bigg(\I_{\{0<x_1+z_1\leq 1\}}  \frac{c_{d,\alpha}}{|x+z|^{d+\alpha}}\bigg).$$
We assume $|x|\leq 1/4$. If $x_1\geq 0$, then
  $$\aligned q(z)\wedge q(x+z)&\geq \I_{\{0<z_1\leq 1-x_1\}} \frac{c_{d,\alpha}}{(|x|+|z|)^{d+\alpha}}\geq \I_{\{0<z_1\leq 1-x_1\}\cap \{|x|\leq |z|\}} \frac{c_{d,\alpha}}{(2|z|)^{d+\alpha}}\\
  &\geq \I_{\{z_1>0\}\cap \{|x|\leq |z|\leq 1-|x|\}} \frac{c_{d,\alpha}}{(2|z|)^{d+\alpha}}.
  \endaligned $$
Therefore, denoting by $S^{d-1}_+ =\{\theta\in \R^d: |\theta|=1 \mbox{ and } \theta_1>0\}$ the half sphere and $\sigma( d\theta)$ the spherical measure, we have
  \begin{align*}
  \int_{\R^d} q(z)\wedge q(x+z)\, dz &\geq \frac{c_{d,\alpha}}{2^{d+\alpha}} \int_{\{z_1>0\}\cap \{|x|\leq |z|\leq 1-|x|\}} \frac{1}{|z|^{d+\alpha}}\, dz\\
  &= \frac{c_{d,\alpha}}{2^{d+\alpha}} \int_{|x|}^{1-|x|} r^{d-1} dr \int_{S^{d-1}_+}\frac{\sigma(d\theta)}{|r\theta|^{d+\alpha}}= \frac{c_{d,\alpha}\, \omega_d}{2^{d+1+\alpha}\, \alpha}\bigg( \frac1{|x|^\alpha}- \frac1{(1-|x|)^\alpha}\bigg),
  \end{align*}
where $\omega_d=\sigma(S^{d-1})$ is the area of the sphere. Since  $|x|\leq 1/4$, it is clear that
  \begin{equation}\label{proof-exa-1}
  \int_{\R^d} q(z)\wedge q(x+z)\, dz\geq \frac{c_{d,\alpha}\, \omega_d}{2^{d+1+\alpha}\, \alpha}\Big(1-\frac1{3^\alpha}\Big)\frac1{|x|^\alpha}.
  \end{equation}

If $x_1<0$, then
  \begin{align*}
  q(z)\wedge q(x+z)&\geq \I_{\{-x_1< z_1\leq 1\}} \frac{c_{d,\alpha}}{(|x+z|+|x|)^{d+\alpha}}\geq \I_{\{-x_1< z_1\leq 1\}\cap \{|x|\leq |x+z|\}} \frac{c_{d,\alpha}}{(2|x+z|)^{d+\alpha}}\\
  &\geq \I_{\{z_1+ x_1>0\}\cap \{|x|\leq |x+z|\leq 1-|x|\}} \frac{c_{d,\alpha}}{(2|x+z|)^{d+\alpha}}.
  \end{align*}
Hence, similar to the argument for the case that $x_1\ge0$, we have
  \begin{align*}
  \int_{\R^d} q(z)\wedge q(x+z)\, dz &\geq \frac{c_{d,\alpha}}{2^{d+\alpha}} \int_{\{z_1+ x_1>0\}\cap \{|x|\leq |x+z|\leq 1-|x|\}} \frac{1}{|x+z|^{d+\alpha}}\, dz\\
  &= \frac{c_{d,\alpha}}{2^{d+\alpha}} \int_{\{z_1 >0\}\cap \{|x|\leq |z|\leq 1-|x|\}} \frac{1}{|z|^{d+\alpha}}\, dz= \frac{c_{d,\alpha}\, \omega_d}{2^{d+1+\alpha}\, \alpha}\Big(1-\frac1{3^\alpha}\Big)\frac1{|x|^\alpha}.
  \end{align*}
Combining this with \eqref{proof-exa-1}, we get that for all $0<s\le 1/4$,
$$J(s)\ge \inf_{|x|=s}  \int_{\R^d} q(z)\wedge q(x+z)\, dz\ge \frac{c_{d,\alpha}\, \omega_d}{2^{d+1+\alpha}\, \alpha}\Big(1-\frac1{3^\alpha}\Big) s^{-\alpha},$$ which finishes the proof.
\end{proof}

\begin{proof}[Proof of Example $\ref{est-stable}$]
According to the proof of Example $\ref{exa}$, we can take $\sigma(r)=a_1r^{1-\alpha}$ and so $g_1(r)=a_2 r^\alpha$ for some $a_1,a_2>0$ in Theorem \ref{thtpw}. Therefore, the required estimates follow from Remark \ref{r:thtpw}(2).
\end{proof}

\subsection{Proofs of results related to strong ergodicity}\label{section4}
Similar to Theorem \ref{th1}(b), Theorem \ref{thm-strong-ergo} is a consequence of the following result.
\begin{theorem}\label{thtvarst}
Assume that the drift $b$ satisfies ${\mathbf B(K_1, \Phi_2(r), l_0)}$ for some $K_1, l_0\ge 0$ and some positive measurable function $\Phi_2$ such that $\Phi_2(r)$ is bounded from below for $r$ large enough and satisfies \eqref{erg-2}, and that \eqref{th10} holds for the L\'{e}vy measure $\nu$ with some $\kappa_0>0$.
Moreover, suppose that there is a nondecreasing and concave function $\sigma\in C([0,2l_0])\cap C^2((0,2l_0])$ such that for some $\kappa\in (0,\kappa_0\wedge l_0]$, one has
  $$
  \sigma(r)\leq \frac1{2r} J(\kappa\wedge r) (\kappa\wedge r)^2, \quad r\in (0, 2l_0];
  $$
and the function $g(r)=\int_{0}^r \frac{ds}{\sigma(s)}$ is well defined for all $r\in [0,2l_0]$. Then there exist constants $\lambda, c>0$ such that for any $x,y\in\R^d$ and $t>0$,
  $$  \|\delta_xP_t- \delta_yP_t\|_{{\rm Var}}\le c e^{-\lambda t}.$$
\end{theorem}

\begin{proof} Without loss of generality, we can and do assume that $l_0\ge 1$ is large enough such that $\inf_{r\ge l_0}\Phi_2(r)>0$ and $\Phi_2$ is increasing on $[l_0,\infty)$; otherwise, we can use $\Phi_2^*(r):=\inf_{s\ge r}\Phi_2(s)$ instead of $\Phi_2(r)$. Define
  $$\psi(r)=\begin{cases}
  c_1 r+ \int_0^r e^{-c_2 g(s)}\, ds ,& r\in [0, 2l_0]; \\
 \psi(2l_0)+ {\psi'(2l_0)}\Phi_2(2l_0)\int_{2l_0}^r\frac{1}{\Phi_2(s)}\,ds, & r\in(2l_0,\infty),
  \end{cases}$$
where $c_1,c_2>0$ are determined later. It is easy to see that $\psi\in C_b^1([0,\infty))$ is concave, due to (2) in Lemma \ref{lem-test-funct} and the increasing property of $\Phi_2$ on $[l_0,\infty)$.  For any $n \ge1$, define $\psi_{n}\in C^1([0,\infty))$ such that $\psi_{n}$ is strictly increasing and
  $$\psi_{n}(r)\begin{cases}
      = \psi(r)        &  0\le r\le 1/(n+1);\\
        \le a+\psi(r), & 1/(n+1)< r\le 1/n;\\
       = a+\psi(r),    &1/n\le r< \infty,
  \end{cases}$$
where $a>0$ is determined below. We still have \eqref{p-var1}, hence the function $\psi_n$ satisfies \eqref{p-w-1} for all $n\ge1$.

Let $\kappa\in (0,\kappa_0\wedge l_0]$ be the constant in the statement of the theorem, and $K_2>0$. On the one hand, take $c_1= e^{-c_2g(2l_0)}$, $c_2=2(K_2+K_1/\kappa)$ and
  $$a=\frac{2}{J_{\kappa}}\left(K_1(c_1+ 1)+\frac{2K_2 c_1}{c_1+ 1} \psi(\kappa)\right),$$ where $J_{\kappa}:=\inf_{0<s\le \kappa} J(s)>0$, thanks to \eqref{th10}.
Using ${\mathbf B(K_1, \Phi_2(r), l_0)}$ and following the arguments of \eqref{sim-prof1} and \eqref{sim-prof2}, we can get that for all $r\in [1/n, l_0]$,
  $$ \Theta_n(r) \le - \frac{2K_2 c_1}{c_1+ 1} \psi(r\vee \kappa)\le -\frac{2K_2 c_1}{c_1+ 1} \psi(\kappa).$$
On the other hand, by ${\mathbf B(K_1, \Phi_2(r), l_0)}$ again, if $r\in (l_0,2l_0]$,
  $$-\Phi_2(r)\psi_n'(r)= -\Phi_2(r) \psi'(r)= -\Phi_2(r) \big(c_1+ e^{-c_2 g(r)}\big)\le -2c_1 \Phi_2(l_0);$$
while for $r>2l_0$,
  \begin{align*}
  -\Phi_2(r)\psi_n'(r)= & -\Phi_2(r) \psi'(r)= -  {\psi'(2l_0)}\Phi_2(2l_0)=-2c_1\Phi_2(2l_0),
  \end{align*}
where the last two equalities follow from the definition of $\psi$. Combining all conclusions above with the fact that $\psi_n$ is uniformly bounded with respect to $n$, $\psi_{n}$ satisfies ${\mathbf C(\lambda, \kappa, n)}$ with some constant $\lambda>0$ for all $n\ge 1$ large enough.

Therefore, by Theorem \ref{p-w}, for any $x,y\in\R^d$,
  \begin{align*}
  \| \delta_xP_t- \delta_y P_t\|_{\var} &\le 2a^{-1} W_{a\I_{(0,\infty)}+\psi}(\delta_xP_t, \delta_y P_t)\le 2e^{-\lambda t}\Big(1+\frac{1}{a}\psi(|x-y|)\Big)\le c e^{-\lambda t}.
  \end{align*}
By now we have proved the desired assertion.
\end{proof}

At the end of this section, we give the

\begin{proof}[Proof of Proposition $\ref{p-ergo}$] Under ${\mathbf B(K_1r, \Phi_2(r), l_0)}$, it holds that for any $x\in\R^d$ with $|x|$ large enough,
  $$\frac{\langle b(x),x\rangle}{|x|}\le -\Phi_2(|x|)+\frac{\langle b(0),x\rangle }{|x|}\le -\frac{1}{2}\Phi_2(|x|),$$
where in the last inequality we have used the fact that $\liminf_{r\to\infty}\frac{\Phi_2(r)}{r}=\infty$. Let $f\in C^2(\R^d)$ such that $f(x)=\log (1+|x|)$ for all $|x|\ge 1$. Then, by \eqref{red}, we can easily establish the following Foster--Lyapunov type condition:
  \begin{equation}\label{ly}
  L_X f(x)\le - c_1 \frac{\Phi_2(|x|)}{1+|x|}+c_2,\quad x\in \R^d,
  \end{equation}
where $L_X$ is the generator of the process $(X_t)_{t\ge0}$ given by \eqref{SDE-generator}, and $c_1,c_2$ are two positive constants. On the other hand, since $b$ satisfies ${\mathbf B(K_1r, \Phi_2(r), l_0)}$ and $\liminf_{r\to\infty}\frac{\Phi_2(r)}{r}=\infty$, $b$ satisfies ${\mathbf B(K_1r, K_2r, l'_0)}$ for some constants $K_2,l_0'>0$, and so Theorem \ref{th1} holds, also thanks to the fact that the associated L\'evy measure $\nu$ satisfies \eqref{th13-11}. Then, there exist constants $\lambda,c>0$ such that for any $x,y\in\R^d$ and $t>0$,
  $$W_1(\delta_x P_t,\delta_yP_t)\le c e^{-\lambda t} |x-y|.$$
This implies that (e.g. see \cite[Theorem 5.10]{Chen1})
  $$\|P_tf\|_{{\rm Lip}}\le c e^{-\lambda t}\|f\|_{{\rm Lip}}$$
holds for any $t>0$ and any Lipschitz continuous function $f$, where $\|f\|_{{\rm Lip}}$ denotes the Lipschitz semi-norm with respect to the Euclidean norm $|\cdot|$. By the standard approximation, we know that the semigroup $(P_t)_{t\ge0}$ is Feller, i.e. for every $t>0$, $P_t$ maps $C_b(\R^d)$ into $C_b(\R^d)$. (Indeed, by \eqref{th13-11} and Corollary \ref{th22-1}(4) below, the semigroup $(P_t)_{t\ge0}$ is strongly Feller, i.e. for every $t>0$, $P_t$ maps $B_b(\R^d)$ into $C_b(\R^d)$, where $B_b(\R^d)$ denotes the class of bounded measurable functions on $\R^d$.) This along with  \eqref{ly}, $\liminf_{r\to\infty}\frac{\Phi_2(r)}{r}=\infty$ and \cite[Theorems 4.5]{MT} yields that the process $(X_t)_{t\ge0}$ has an invariant probability measure.

Furthermore, under the assumptions Theorem \ref{thm-strong-ergo} holds. Then, we can deduce from \eqref{thm-strong-ergo-1} that the process $(X_t)_{t\ge0}$ has at most one invariant probability measure, so by the above arguments, it admits a unique one. Indeed, let $\mu_1$ and $\mu_2$ be invariant probability measures of the process $(X_t)_{t\ge0}$. Then,
  $$\|\mu_1-\mu_2\|_{{\rm Var}}=\iint\|\delta_xP_t-\delta_yP_t\|_{{\rm Var}}\,\mu_1(dy)\,\mu_2(dx)\le ce^{-\lambda t}.$$
Letting $t\to\infty$, we find that $\mu_1=\mu_2$. Denote by $\mu$ the unique invariant probability measure. Therefore, by \eqref{thm-strong-ergo-1}, we have
  $$\|\delta_xP_t-\mu\|_{{\rm Var}}\le \int  \|\delta_xP_t-\delta_yP_t\|_{{\rm Var}}\,\mu(dy)\le ce^{-\lambda t}.$$
The proof is complete.
 \end{proof}

\section{Further applications of the refined basic coupling}

\subsection{Spatial regularity of semigroups}\label{section5}

As another application of the refined basic coupling for L\'evy processes, we shall study in this subsection the regularity of the semigroup $(P_t)_{t\ge0}$ for SDEs with L\'{e}vy noises, a topic which has attracted lots of interests in recent years. For instance, the Bismut--Elworthy--Li's derivative formula and gradient estimates for SDEs driven by (multiplicative) L\'{e}vy noise have been established in \cite{Zhang13, WXZ}. Note that, when the L\'{e}vy noise is reduced to a symmetric $\alpha$-stable process, the statement of Corollary \ref{th22-1} below is weaker than those in \cite{Zhang13, WXZ}; however, it works for more general L\'{e}vy noises. Besides, the drift term $b$ in our setting only satisfies the one-sided Lipschitz condition; while in \cite{Zhang13, WXZ} it is required to be in $C_b^1(\R^d)$, which is essentially due to the fact that the Malliavin calculus was used there.

Throughout this part, we assume that \eqref{th10} holds for the L\'{e}vy measure $\nu$ with some $\kappa_0>0$, and the drift term $b$ satisfies the following one-sided Lipschitz condition, i.e.\ there is a constant $ K_1>0$ such that for any $x,y\in\R^d$,
  $$\langle b(x)-b(y),x-y\rangle\le K_1|x-y|^2.$$

\begin{theorem}\label{p-regg} Assume that \eqref{th10} holds and $b$ satisfies the one-sided Lipschitz condition.
For some fixed $\varepsilon_0\in(0,\kappa_0]$, let $\phi\in C^1([0,2\varepsilon_0])$ be such that $\phi(0)=0$, $\phi'\ge 0$, and for all $0<\varepsilon\le \varepsilon_0$
$$A_\varepsilon (\phi):=\inf_{0<r\le \varepsilon} \Big\{ \frac{1}{2}J(r)(2\phi(r)-\phi(2r))-K_1 \phi'(r) r  \Big\}>0.$$
Then, for any $f\in B_b(\R^d)$ and $t>0$,
  \begin{equation}\label{e:reg}
  \sup_{x\neq y}\frac{{|P_t f(x)-P_t f(y)|}}{\phi(|x-y|)}
  \le 2\|f\|_\infty\inf_{\varepsilon\in(0,\varepsilon_0]}
  \bigg[\frac{1}{\phi(\varepsilon)}+\frac{1}{tA_{\varepsilon}(\phi)}\bigg].
  \end{equation}\end{theorem}

\begin{proof} Let $\widetilde{L}=\widetilde{L}_X $ be the coupling operator given in \eqref{SDE-coup-op}.
For any $x,y\in\R^d$ with $0<|x-y|\le \varepsilon \le \varepsilon_0$, by applying \eqref{proofth2544} with $\kappa=\kappa_0$ and noticing that $\varepsilon_0\le \kappa_0$, we have
  \begin{equation}\label{e:sregu}
  \begin{split}
  \widetilde L\phi(|x-y|)
  &=\frac{1}{2} \mu_{x-y}(\R^d)\big[\phi(2|x-y|)-2\phi(|x-y|)\big]+ K_1{\phi'(|x-y|)}{|x-y|}\\
  &\le \frac{1}{2}J(|x-y|) \big[\phi(2|x-y|)-2\phi(|x-y|)\big]+ K_1\phi'(|x-y|)|x-y|\\
  &\le-A_\varepsilon (\phi)<0.
  \end{split}
  \end{equation}

Below we follow the same argument as in the proof of \cite[Theorem 1.2]{LW14}. We still use the coupling process $(X_t,Y_t)_{t\ge0}$ constructed
in Section \ref{sectioncousde},  and denote by $\widetilde{\Pp}^{(x,y)}$ and
$\widetilde{\Ee}^{(x,y)}$ the distribution and the expectation of
$(X_t,Y_t)_{t\ge0}$ starting from $(x,y)$, respectively. For any
$n\ge1$ and $\varepsilon\in(0,\varepsilon_0]$, we set
  $$\aligned S_\varepsilon&:=\inf\{t\ge0: |X_t-Y_t|>\varepsilon\},\\
  T_n&:=\inf\{t\ge0: |X_t-Y_t|\le 1/n\},\\
  T_{n,\varepsilon}&:=T_n\wedge S_\varepsilon.\endaligned$$
Furthermore, we still use the coupling time defined by
 $$T:=\inf\{t\ge0: X_t=Y_t\}.$$
Note that $T_n\uparrow T$ as $n\uparrow\infty.$
For any $x,$ $y\in\R^d$ with $0<|x-y|<\varepsilon\le \varepsilon_0$, we take $n$ large enough such that $|x-y|>1/n$. Then, by \eqref{e:sregu},
  $$\aligned
  0&\le\widetilde{\Ee}^{(x,y)}\phi\big(|X_{t\wedge T_{n,\varepsilon}}-Y_{t\wedge T_{n,\varepsilon}}|\big)=\phi(|x-y|)+\widetilde{\Ee}^{(x,y)}\bigg(\int_0^{t\wedge T_{n,\varepsilon}} \widetilde{L} \phi\big(|X_{s}-Y_{s}|\big)\,ds\bigg)\\
  &\le \phi(|x-y|)-A_\varepsilon (\phi)\widetilde{\Ee}^{(x,y)}(t\wedge
  T_{n,\varepsilon}).\endaligned$$
Therefore
  $$\widetilde{\Ee}^{(x,y)}(t\wedge T_{n,\varepsilon})\le
  \frac{ \phi(|x-y|)}{ A_\varepsilon(\phi)}.$$
Letting $t\rightarrow\infty$ and then $n\rightarrow\infty$, we
arrive at
  \begin{equation}\label{e:sregu-1}
  \widetilde{\Ee}^{(x,y)}(T_{}\wedge S_\varepsilon) \le \frac{ \phi(|x-y|)}{ A_\varepsilon(\phi)}.
  \end{equation}

On the other hand, again by \eqref{e:sregu}, for any $x$, $y\in\R^d$ with $1/n\le |x-y|<\varepsilon\le \varepsilon_0$,
  $$\aligned \widetilde{\Ee}^{(x,y)}\phi\big(|X_{t\wedge T_{n,\varepsilon}}-Y_{t\wedge T_{n,\varepsilon}}|\big)=\phi(|x-y|)+\widetilde{\Ee}^{(x,y)}\bigg(\int_0^{t\wedge T_{n,\varepsilon}}  \widetilde{L} \phi(| X_u-Y_u|)\,du\bigg)\le \phi(|x-y|),\endaligned$$
which yields that
  $$\phi(\varepsilon)\widetilde{\Pp}^{(x,y)}(S_\varepsilon<T_n\wedge t)\le \phi(|x-y|).$$
Letting $t\rightarrow\infty$ and then $n\to \infty$ leads to
  \begin{equation}\label{e:sregu-2}
  \widetilde{\Pp}^{(x,y)}(T>S_\varepsilon)\le\frac{ \phi(|x-y|)}{\phi(\varepsilon)}.
  \end{equation}

Therefore, for any $x$, $y\in\R^d$ with $0<|x-y|<\varepsilon\le \varepsilon_0$, by \eqref{e:sregu-1} and \eqref{e:sregu-2},
  $$\aligned
   \widetilde{\Pp}^{(x,y)}(T> t)
    &\le \widetilde{\Pp}^{{(x,y)}}(T\wedge S_\varepsilon>t)+\widetilde{\Pp}^{{(x,y)}}(T>S_\varepsilon)\\
  &\le \frac{\widetilde{\Ee}^{(x,y)}(T\wedge S_\varepsilon)}{t}+\frac{\phi(|x-y|)}{\phi(\varepsilon)}\le \phi(|x-y|)\bigg[\frac{1}{\phi(\varepsilon)}+\frac{1}{tA_{\varepsilon}(\phi)}\bigg].
  \endaligned$$
Hence, for any $f\in B_b(\R^d)$, $t>0$ and any $x$, $y\in\R^d$ with $0<|x-y|<\varepsilon\le \varepsilon_0$,
  \begin{align*}
  {|P_t f(x)-P_t f(y)|}&={|\Ee^xf(X_t)-\Ee^yf(Y_t)|}\\
  &={\big|\widetilde{{\Ee}}^{(x,y)}(f(X_t)-f(Y_t))\big|}={\big|\widetilde{{\Ee}}^{(x,y)}(f(X_t)-f(Y_t))\I_{\{T> t\}}\big|}\\
  &\le 2\|f\|_\infty {\widetilde{{\Pp}}^{(x,y)}(T> t)}\le 2\|f\|_\infty\phi(|x-y|)\bigg[\frac{1}{\phi(\varepsilon)}+\frac{1}{tA_{\varepsilon}(\phi)}\bigg].
  \end{align*}
As a result,
  $$\sup_{|x-y|\le\varepsilon}\frac{{|P_t f(x)-P_t f(y)|}}{\phi(|x-y|)}
  \le 2\|f\|_\infty\bigg[\frac{1}{\phi(\varepsilon)}+\frac{1}{tA_{\varepsilon}(\phi)}\bigg].$$
This along with the fact that
  $$\sup_{|x-y|\ge\varepsilon}\frac{{|P_t f(x)-P_t f(y)|}}{\phi(|x-y|)}   \le \frac{2\|f\|_\infty}{\phi(\varepsilon)}$$
further gives us that for all $\varepsilon\in(0,\varepsilon_0]$,
  $$\sup_{x\neq y}\frac{{|P_t f(x)-P_t f(y)|}}{\phi(|x-y|)}\le
  2\|f\|_\infty \bigg[\frac{1}{\phi(\varepsilon)}+\frac{1}{tA_{\varepsilon}(\phi)}\bigg]. $$
The desired assertion follows from the inequality above by taking infimum with respect to $\varepsilon\in(0,\varepsilon_0]$ in the right hand side.
\end{proof}

As a consequence of Theorem \ref{p-regg}, we have the following result.
\begin{proposition}\label{th22}
Assume that \eqref{th10} holds and $b$ satisfies the one-sided Lipschitz condition. If there exist a constant $\varepsilon_0 \in(0,\kappa_0]$ and  a function $\phi\in C^3([0,2\varepsilon_0])$ such that $\phi(0)=0$, $\phi'\ge0$, $\phi''\le 0$ and $\phi'''\ge0$, and that
  \begin{equation}\label{e:th22}
  \lim_{\varepsilon\to0}\sup_{0<r\le \varepsilon} J(r) r^2 \phi''(2r)<0,
  \end{equation}
then there are constants $C>0$ and $\varepsilon'_0\in(0,\varepsilon_0]$ such that for any $f\in B_b(\R^d)$ and $t>0$,
  \begin{equation}\label{e:th2e1}
  \sup_{x\neq y}\frac{{|P_t f(x)-P_t f(y)|}}{\phi(|x-y|)}
  \le C\|f\|_\infty\inf_{\varepsilon\in(0,\varepsilon'_0]}
  \bigg[\frac{1}{\phi(\varepsilon)}+\frac{1}{tB_{\varepsilon}(\phi)}\bigg],
  \end{equation} where $$  B_{\varepsilon}(\phi):=-\sup_{0<r\le \varepsilon} J(r) r^2 \phi''(2r).$$
\end{proposition}

\begin{proof}
Since $\phi'''\geq 0$, we have
  $$2\phi(r)-\phi(2r)=-\int_0^r\int_s^{r+s} \phi''(u)\,du\,ds\ge -\phi''(2r) r^2.$$
On the other hand, by $\phi''\le 0$ and the fact that $\phi(0)=0$,
  $$\phi'(r)r\le \int_0^r \phi'(s)\,ds =\phi(r).$$
Therefore,
  $$\aligned \frac12 J(r)(2\phi(r)-\phi(2r))- K_1\phi'(r)r &\geq -\frac12 J(r)r^2 \phi''(r) -K_1\phi(r) \geq \frac12 B_\varepsilon(\phi) -K_1\phi(r). \endaligned$$
According to \eqref{e:th22}, we know that there is a constant $\varepsilon'_0\in(0,\varepsilon_0]$ such that for all $\varepsilon\in(0,\varepsilon_0']$,
  $$A_\varepsilon (\phi)\ge \frac{1}{4} B_\varepsilon(\phi)>0.$$
Then, the desired assertion \eqref{e:th2e1} follows immediately from Theorem \ref{p-regg}.
\end{proof}

Furthermore, we have the following more explicit regularity properties of the semigroup $(P_t)_{t\ge0}$.

\begin{corollary}\label{th22-1}
Assume that \eqref{th10} holds for some $\kappa_0>0$ and $b$ satisfies the one-sided Lipschitz condition.
\begin{itemize}
\item[(1)]
If for some $\theta>0$, $$\lim_{\varepsilon\to0}\inf_{0<r\le \varepsilon} J(r) r \left(\log\frac{1}{r}\right)^{-(1+\theta)}>0,$$ then there exist constants $C>0$ and $\varepsilon'_0 \in (0,\kappa_0]$ such that for all $f\in B_b(\R^d)$ and $t>0$,
$$
  \sup_{x\neq y}\frac{{|P_t f(x)-P_t f(y)|}}{|x-y|}
  \le C\|f\|_\infty\inf_{\varepsilon\in(0,\varepsilon'_0]}
  \bigg[\frac{1}{\varepsilon}+\frac{1}{t\inf_{0<r\le \varepsilon} J(r) r\left(\log\frac{1}{r}\right)^{-(1+\theta)}}\bigg].
  $$

\item[(2)]
If for some $\theta>0$,
  $$\lim_{\varepsilon\to0}\inf_{0<r\le \varepsilon} J(r) r\left(\log\frac{1}{r}\right)^{\theta-1} >0,$$
then there exist constants $C>0$ and $\varepsilon_0'\in(0,\kappa_0]$ such that for all $f\in B_b(\R^d)$ and $t>0$,
$$
  \sup_{x\neq y}\frac{{|P_t f(x)-P_t f(y)|}}{|x-y|\, \big|\!\log|x-y|\big|^\theta}
  \le C\|f\|_\infty\inf_{\varepsilon\in(0,\varepsilon_0']}
  \bigg[\frac{1}{\varepsilon |\log \varepsilon|^\theta}+\frac{1}{t\inf_{0<r\le \varepsilon} J(r) r\left(\log\frac{1}{r}\right)^{\theta-1}}\bigg].
  $$

\item[(3)] If for some $\theta\in(0,1)$
  $$\lim_{\varepsilon\to0}\inf_{0<r\le \varepsilon} J(r) r^\theta>0,$$
then there exist constants $C>0$ and $\varepsilon_0' \in(0,\kappa_0]$ such that for all $f\in B_b(\R^d)$ and $t>0$,
$$
  \sup_{x\neq y}\frac{{|P_t f(x)-P_t f(y)|}}{|x-y|^\theta}
  \le C\|f\|_\infty\inf_{\varepsilon\in(0,\varepsilon_0']}
  \bigg[\frac{1}{\varepsilon^\theta}+\frac{1}{t\inf_{0<r\le \varepsilon} J(r) r^\theta}\bigg].
  $$

\item[(4)] If for some $\theta>0$,
  $$\lim_{\varepsilon\to0}\inf_{0<r\le \varepsilon} J(r)\Big(\log \frac{1}{r}\Big)^{-(1+\theta)}>0,$$
then there exist constants $C>0$ and $\varepsilon'_0\in(0,\kappa_0]$ such that for all $f\in B_b(\R^d)$ and $t>0$,
$$
  \sup_{x\neq y}\frac{{|P_t f(x)-P_t f(y)|}}{\big|\log |x-y|\big|^{-\theta}}
  \le C\|f\|_\infty\inf_{\varepsilon\in(0,\varepsilon_0']}
  \bigg[\frac{1}{|\log\varepsilon|}+\frac{1}{t\inf_{0<r\le \varepsilon} J(r) \left(\log\frac{1}{r}\right)^{-(1+\theta)}}\bigg].
  $$
\end{itemize}
\end{corollary}
\begin{proof}
The assertions follow from Proposition \ref{th22} by taking $\phi(r)=r(1-\log^{-\theta}(1/r))$, $\phi(r)=r\log^\theta(1/r)$, $\phi(r)=r^{\theta}$ and $\phi(r)=\log^{-\theta}(1/r)$ for $r>0$ small enough, respectively.
\end{proof}

\subsection{Some discussions on results related to $L^1$-Wasserstein
distance}\label{5.3}

In this part, motivated by \cite[Subsections 2.3, 2.4 and Section
3]{Eberle}, we discuss some variations of Theorem \ref{thtpw},
including perturbations of the drift, local versions and  generalizations to
product spaces. Since the estimates in Theorem \ref{thtpw} are
explicit, based on the results in this subsection, applications to
overdamped Langevin SDEs with jumps and systems of weakly
interacting SDEs with jumps can be treated in the same way as in
\cite{Eberle}.

\subsubsection{Perturbations of the drift}
Let
  \begin{equation}\label{per}
  b^\ast(x)=b(x)+\gamma(x),\quad x\in \R^d,
  \end{equation}
where $b$ satisfies condition (b) in Theorem \ref{thtpw}. Let $l_0$ be the constant given there. We now suppose that $\gamma$ satisfies the assumption ${\mathbf B(\Phi_\gamma(r),0, l_0)}$ with some nonnegative measurable function $\Phi_{\gamma}$, i.e.\  for any $x,y\in\R^d$,
  $$\frac{\langle \gamma(x)-\gamma(y), x-y\rangle}{|x-y|}\leq
  \begin{cases}
  \Phi_\gamma(|x-y|), & |x-y|< l_0;\\
  0, & |x-y|\ge l_0.
  \end{cases} $$
Moreover, we suppose that $\Phi_\gamma$ also satisfies condition (b) in Theorem \ref{thtpw}, i.e.\  $\Phi_\gamma \in C([0,2l_0])\cap C^2((0,2l_0])$ is a nonnegative concave function such that $\Phi_\gamma(0)=0$ and $\Phi''_\gamma$ is nondecreasing.
 Then, the following perturbation result for Theorem \ref{thtpw} holds.

\begin{proposition}\label{P:per}
Let the drift $b^\ast$ be given by \eqref{per} with $b$ and $\gamma$ satisfying the assumptions stated above, and let $(P_t^\ast)_{t\ge 0}$ be the transition semigroup corresponding to the SDE \eqref{s1} by replacing $b$ with $b^\ast$. Under conditions $(a)$ and $(c)$ in Theorem $\ref{thtpw}$, if the function $g_\gamma(r)= \frac2{c_2} \int_0^r \frac{\Phi_\gamma(s)}{s \sigma(s)}\,ds$ is well defined for all $r\in [0,2l_0]$, then there is a constant $C^\ast>0$ such that for  any $x,y\in\R^d$ and $t>0$,
  $$  W_1(\delta_x P^\ast_t, \delta_y P^\ast_t)\le C^\ast e^{-\lambda^\ast t} |x-y|,$$
where
  $$\lambda^\ast=\frac{(2K_2)\wedge g_1(2l_0)^{-1}}{1+ \exp\big\{ \big[g(2l_0)+g_\gamma(2l_0)\big] \big[ (2K_2)\wedge g_1(2l_0)^{-1} \big] \big\}},$$
and $g_1$ and $g$ are defined in Theorem $\ref{thtpw}$. In particular,
  $$\lambda^\ast\ge \lambda\, e^{-g_\gamma(2l_0)[ (2K_2)\wedge g_1(2l_0)^{-1} ]} ,$$
where $\lambda$ is given by \eqref{pwcon}.
\end{proposition}

\begin{proof}
It is easy to see that, under the assumptions of the proposition, the drift $b^\ast$ satisfies ${\mathbf B(\Phi_1(r)+ \Phi_\gamma(r), K_2r, l_0)}$. Then, the desired assertion follows from Theorem \ref{thtpw}.
\end{proof}

\subsubsection{Local exponential contractivity}

Sometimes, the drift $b$ does not satisfy \eqref{th111}, but the process $(X_t)_{t\ge0}$ will stay with high probability inside some ball for a long time. Similar to \cite[Theorem 6]{Eberle}, we are still able to obtain the exponential contractivity up to an error term determined by the exit probabilities of the process $(X_t)_{t\ge0}$ from the ball.

In the following, we assume that for some fixed $R>0$, there are constants
$l_0:=l_0(R)\in [0,R]$, $K_2:=K_2(R)>0$ and a nonnegative measurable
function $\Phi_1:=\Phi_{1,R}$ satisfying  condition (b) in Theorem
\ref{thtpw} such that for any $x,y\in\R^d$ with $|x-y|\leq R$,
  $$ \frac{\langle b(x)-b(y), x-y\rangle}{|x-y|}\le \Phi_1(|x-y|)-\big[\Phi_1(|x-y|) +K_2|x-y| \big]\I_{\{l_0<|x-y|\le R\}}.$$
Then, we have the following statement for local exponential contractivity.

\begin{proposition}\label{local}
Suppose that the drift $b$ satisfies the assumptions above for some $R>0$, and assume that conditions $(a)$ and $(c)$ in Theorem $\ref{thtpw}$ are also fulfilled. Then, for any $t>0$ and $x,y\in\R^d$,
  $$W_{\psi_R}(\delta_xP_t,\delta_yP_t)\le e^{-\lambda t}\psi_R(|x-y|)+\psi_R(R)\big[\Pp^x(\tau_{B(0,R/2)}\le t)+\Pp^y(\tau_{B(0,R/2)}\le t)\big],$$
where $\psi_R(r)=\psi(r\wedge R)$,
$\tau_{B(0,R)}=\inf\{t>0:|X_t|>R\}$, and $\psi$ and $\lambda$ the
function and the constant defined in Theorem $\ref{thtpw}$
respectively.
\end{proposition}

\begin{proof}
Following the argument of Theorem \ref{thtpw}, we have
  $$\widetilde L \psi_{R}(|x-y|)\le -\lambda  \psi_R(|x-y|),\quad 0<|x-y|\le R,$$
where $\widetilde{L}=\widetilde{L}_X $ is the coupling operator
given in \eqref{SDE-coup-op}.  Let $(X_t,Y_t)_{t\ge0}$ be the
coupling process constructed in Subsection \ref{sectioncousde}.
According to the proof of Theorem \ref{p-w}, we can get that for any
$x,y\in\R^d$ with $0<|x-y|\leq R$ and any $t\ge0$,
  $$\widetilde{\Ee}^{(x,y)}[e^{\lambda(t\wedge\tilde\tau_R)} \psi_R(|X_{t\wedge \tilde\tau_R}-Y_{t\wedge \tilde\tau_R}|)]\leq \psi_R(|x-y|),$$
where $\tilde \tau_R=\inf\{t>0:|X_t-Y_t|\ge R\}.$ Therefore,
  \begin{align*}
  \widetilde{\Ee}^{(x,y)}\psi_R(|X_t-Y_t|) &\le \widetilde{\Ee}^{(x,y)}(\psi_R(|X_t-Y_t|):\tilde \tau_R>t)+ \psi(R)\widetilde\Pp^{(x,y)}(\tilde \tau_R\le t)\\
  &\le e^{-\lambda t}\widetilde{\Ee}^{(x,y)}[e^{\lambda(t\wedge\tilde\tau_R)} \psi_R(|X_{t\wedge \tilde\tau_R}-Y_{t\wedge \tilde\tau_R}|)] \\
  &\hskip14pt + \psi(R)\big(\Pp^x(\tau_{B(0,R/2)}\le t)+\Pp^y(\tau_{B(0,R/2)}\le t)\big)\\
  &\le \psi_R(|x-y|)e^{-\lambda t}+\psi(R)\big(\Pp^x(\tau_{B(0,R/2)}\le t)+\Pp^y(\tau_{B(0,R/2)}\le t)\big),
  \end{align*}
which proves the desired assertion.
\end{proof}

As an application of Proposition \ref{local}, we can consider the
local exponential contractivity for the equation \eqref{s1} in $\R^{d-1}$, with $b(x)=(b^1(x),\ldots, b^{d-1}(x))$ given by
  $$b^i(x)=d^2(x^{i+1}-2x^i+x^{i-1})+V'(x^i),\quad i=1,\ldots,d-1,$$
and $x^0=x^d=0$, where $x=(x^1, x^2, \ldots, x^{d-1})\in\R^{d-1}$ and $V:\R\to \R$ is a  $C^2$-function such that $V''\ge -L$ for some finite constant $L\in\R$. Such an equation is called a stochastic heat equation with jumps. See the related discussions for diffusions in \cite[Example 5]{Eberle}.

\subsubsection{Exponential contractivity on product spaces}

We consider a system
  \begin{equation}\label{interlevy}
  dX_t^i=b^i(X_t)+dZ_t^i,\quad X^i_0=x^i,\quad i=1,\ldots,n
  \end{equation}
of $n$ interacting SDEs with jumps. Here $(Z_t^i)_{t\ge0}$ $(i=1,\ldots,n)$ are independent L\'evy processes in $\R^{d_i}$, and $b^i:\R^d\to \R^{d_i}$ $(i=1,\ldots,n)$ are measurable functions with $d=\sum_{i=1}^nd_i$. We assume that the system \eqref{interlevy} has a unique strong solution, and denote by $X=(X_t)_{t\ge0}$ this unique solution.

In the following, we assume that
  $$b^i(x)=b_0^i(x^i)+\gamma^i(x),\quad i=1, \ldots,n,$$
where for any $1\le i\le n$, $b^i_0:\R^{d_i}\to \R^{d_i}$ is a
measurable function, and $\gamma^i:\R^d\to \R^{d_i}$ is a
sufficiently small perturbation. If $\gamma^i\equiv0\,
(i=1,\ldots,n)$, then the components $X^1, \ldots, X^n$ of the
process $X$ are independent. To study contraction properties of the
process $X$, we follow \cite[Section 3.1]{Eberle} and consider the
distance function on $\R^d$ of the form
  $$d_{\psi,w}(x,y)=\sum_{i=1}^n w_i \psi_i(|x^i-y^i|),$$
where $w_i\in(0,1]$ are positive weights, and $\psi_i:[0,\infty)\to [0,\infty)$ are strictly increasing concave $C^1$-functions with $\psi_i(0)=0$. In many applications, one can choose $w_i=1$ for all $i=1,\ldots,n$. The corresponding distance will then be denoted by $d_{\psi,1}$. In particular,
  $$d_{l^1}(x,y):=\sum_{i=1}^n|x^i-y^i|.$$
We shall denote by $W_{\psi,w}$ and $W_{l^1}$ the Wasserstein distances on $\mathscr P(\R^d)$ corresponding to $d_{\psi,w}$ and $d_{l^1}$, respectively.

For any $1\le i\le n$, let $\nu_i$ be the L\'evy measure for the L\'evy process $(Z_t^i)_{t\ge0}$. We assume that $\nu_i$ and $b_0^i$ satisfy conditions (a), (b) and (c) in Theorem \ref{thtpw}. Let $\psi_i$, $c_i=c_{1,i}$ and $\lambda_i$ be the corresponding function and constants defined in Theorem \ref{thtpw}, respectively. Note that, $ \psi_i'(0)>0$ and $\psi_i'(r)\le \psi_i'(0)$ for all $r\ge0$ and $1\le i\le n$.

Now, we can state the main result in this part.

\begin{proposition}\label{P:prod}
Under all the assumptions above, if there exist constants $\varepsilon_i\in[0,\lambda_i),\, 1\le i\le n$ such that for any $x,y\in\R^d$,
  \begin{equation}\label{con-remin}
  \sum_{i=1}^n w_i\psi_i'(0)|\gamma^i(x)-\gamma^i(y)|\le \sum_{i=1}^n \varepsilon_i w_i\psi_i(|x^i-y^i|),
  \end{equation}
then for any $t\ge0$ and $x,y\in \R^d$,
  $$W_{\psi, w}(\delta_xP_t, \delta_y P_t)\le e^{-\lambda t}d_{\psi,w}(x,y)$$
and
  $$W_{l^1}(\delta_xP_t, \delta_y P_t)\le C e^{-\lambda t}d_{l^1}(x,y),$$
where
  $$\lambda=\min_{1\le i\le n} (\lambda_i-\varepsilon_i),\quad C=\frac{\max_{1\le i\le n}(c_i+1)}{\min_{1\le i\le n} c_iw_i}.$$
 \end{proposition}

We need some preparations for the proof, especially the coupling on
the product space. For any $x=(x^1,x^2,\ldots, x^n)\in \R^d$ and
$z^i\in \R^{d_i}$ with $1\le i\le n$, we write $x+z^i$ for
$(x^1,\ldots, x^{i-1}, x^i+z^i, x^{i+1},\ldots, x^n)$. Then the
generator of the process $X$ acting on $C_b^2(\R^d)$ is given by
  $$Lf(x)=\langle b(x),\nabla f(x)\rangle +\sum_{i=1}^n \int_{\R^{d_i}}\Big(f(x+z^i)-f(x)-\langle\nabla_{x^i} f(x), z^i\rangle\I_{\{|z^i|\le 1\}}\Big)\,\nu_i(dz^i),$$
where $\nabla_{x^i} f(x)$ is the partial gradient of $f(x)$. For any
$1\leq i\leq n$, let $\mu^i_{x^i}=\nu_i\wedge (\delta_{x^i}\ast
\nu_i)$ and $\kappa^i\in (0,\kappa_0^i]$, where $\kappa_0^i$ is the
constant in the condition \eqref{th10} for $\nu_i$. Motivated by
\eqref{proofth24}, we define the operator $\widetilde L$ as follows:
for any $h\in C_b^2(\R^d\times \R^d)$,
  \begin{align*}
  \widetilde L h(x,y)  &= \<b(x),\nabla_x h(x,y)\> + \<b(y),\nabla_y h(x,y)\>\\
  &\quad +\sum_{i=1}^n \bigg[\frac{1}{2}\int_{\R^{d_i}}\! \Big( h(x+z^i,y+ z^i+(x^i-y^i)_{\kappa^i})-h(x,y)-\langle\nabla_{x^i} h(x,y), z^i\rangle \I_{\{|z^i|\le 1\}}\cr
    &\hskip65pt -\langle\nabla_{y^i} h(x,y), z^i+(x^i-y^i)_{\kappa^i}\rangle \I_{\{|z^i+(x^i-y^i)_{\kappa^i}|\le 1\}}\Big)\,\mu^i_{(y^i-x^i)_{\kappa^i}}(dz^i)\cr
    &\hskip45pt +\frac{1}{2}\int_{\R^{d_i}}\! \Big( h(x+z^i,y+ z^i+(y^i-x^i)_{\kappa^i})-h(x,y)-\langle\nabla_{x^i} h(x,y), z^i\rangle \I_{\{|z^i|\le 1\}}\\
    &\hskip65pt -\langle\nabla_{y^i} h(x,y), z^i+(y^i-x^i)_{\kappa^i}\rangle \I_{\{|z^i+(y^i-x^i)_{\kappa^i}|\le 1\}}\Big)\,\mu^i_{(x^i-y^i)_{\kappa^i}}(dz^i)\cr
    &\hskip45pt +\int_{\R^{d_i}}\! \Big( h(x+z^i,y+z^i)-h(x,y)-\langle\nabla_{x^i} h(x,y), z^i\rangle \I_{\{|z^i|\le 1\}}\cr
    &\hskip65pt -\langle\nabla_{y^i} h(x,y), z^i\rangle \I_{\{|z^i|\le 1\}}\Big)\,\Big(\nu_i -\frac{1}{2}\mu^i_{(x^i-y^i)_{\kappa^i}}
    -\frac{1}{2}\mu^i_{(y^i-x^i)_{\kappa^i}}\Big)(dz^i)\bigg].
  \end{align*}
Following the arguments at the end of Subsection \ref{coupling-op},
it is easy to show that $\widetilde L$ is indeed a coupling operator
of the generator $L$. Note that if $h(x,y)=\sum_{i=1}^n
h_i(x^i,y^i)$, where $h_i\in C_b^2(\R^{d_i} \times \R^{d_i})\,
(1\leq i\leq n)$, then
  \begin{equation}\label{coupling-product-1}
  \widetilde L h(x,y)= \sum_{i=1}^n \Big[\<\gamma^i(x),\nabla_{x^i} h_i(x^i,y^i)\> + \<\gamma^i(y),\nabla_{y^i} h_i(x^i,y^i)\> +\widetilde L_i h_i(x^i,y^i)\Big],
  \end{equation}
where $\widetilde L_i$ is the refined basic coupling operator of the
generator corresponding to the following SDE
  $$d\bar{Y}^i_t= b^i_0(\bar{Y}^i_t)\,dt+dZ^i_t,\quad 1\leq i\leq n.$$

Next we shall present the coupling equation of \eqref{interlevy}
corresponding to the coupling operator $\widetilde L$ given above.
For $1\le i\le n$, let $$\rho_i(x^i,z^i)= \frac{\mu^i_{x^i}(dz^i)}
{\nu_i(dz^i)}\in[0,1]$$ and $\bar N_i(dt,dz^i,du)$ be the Poisson
random measure defined in \eqref{Poisson-random-measure} associated
to $(Z^i_t)_{t\geq 0}$. Consider the equations for $1\le i\le n$
which are similar to \eqref{coup-SDE-2}:
  \begin{equation}\label{coupling-product-1.5}
  \begin{split}
  d Y^i_t&= b^i(Y_t)\,d t\\
  &\quad+ \int_{\R^{d_i}\times [0,1]} \Big[\big(z^i+(U^i_{t-})_{\kappa^i}\big) \I_{\{u\le \frac12 \rho_i((-U^i_{t-})_{\kappa^i},z^i)\}} \\
  &\hskip35pt +\big(z^i+(-U^i_{t-})_{\kappa^i}\big) \I_{\{\frac12 \rho_i((-U^i_{t-})_{\kappa^i},z^i)< u\le \frac12 [\rho_i((-U^i_{t-})_{\kappa^i},z^i) +\rho_i((U^i_{t-})_{\kappa^i},z^i)]\}}\\
  &\hskip35pt  + z^i \I_{\{\frac12 [\rho_i((-U^i_{t-})_{\kappa^i},z^i) +\rho_i((U^i_{t-})_{\kappa^i},z^i)]< u\le 1\}}\Big]  \bar{N}_i(dt,dz^i,du)\\
  &\quad -\! \int_{\R^{d_i}\times [0,1]} \!
  \Big[\! \big(z^i \!+\! (U^i_{t-})_{\kappa^i}\big)\! \big(\I_{\{|z^i+(U^i_{t-})_{\kappa^i}|\le 1\}} \! -\!\I_{\{|z^i|\le 1\}}\big)\! \I_{\{u\le \frac12 \rho_i((-U^i_{t-})_{\kappa^i},z^i)\}}\\
  &\hskip35pt +\big(z^i+ (-U^i_{t-})_{\kappa^i}\big)\big(\I_{\{|z^i+(-U^i_{t-})_{\kappa^i}|\le 1\}} -\I_{\{|z^i|\le 1\}}\big)\\
  &\hskip45pt \times\! \I_{\{\frac12 \rho_i((-U^i_{t-})_{\kappa^i},z^i)< u\le
  \frac12 [\rho_i((-U^i_{t-})_{\kappa^i},z^i) + \rho_i((U^i_{t-})_{\kappa^i},z^i)]\}}\Big] \,\nu_i(dz^i)\,du\,dt,
  \end{split}
  \end{equation}
where $U^i_t=X^i_t-Y^i_t$. Similar to the discussions in Subsection
\ref{cou:pro}, the above equation can be simplified as
  \begin{equation}\label{coupling-product-2}
  \begin{split}
  d Y^i_t&= b^i(Y_t)\,d t+ dZ^i_t + \int_{\R^{d_i}\times [0,1]} V^i_{t-}(z^i,u) \bar{N}_i(dt,dz^i,du),\quad Y^i_0=y^i,
  \end{split}
  \end{equation}
where
  $$\aligned
  V^i_{t}(z^i,u)&=(U^i_{t})_{\kappa} \big[ \I_{\{u\le \frac12 \rho_i((-U^i_{t})_{\kappa^i},z^i)\}} - \I_{\{\frac12 \rho_i((-U^i_{t})_{\kappa^i},z^i) < u\le \frac12 [\rho_i((-U^i_{t})_{\kappa^i},z^i) + \rho_i((U^i_{t})_{\kappa^i},z^i)] \}}\big].
  \endaligned$$
The following result is analogous to those in Propositions \ref{2-prop-1} and \ref{3-prop-2}.

\begin{lemma}\label{coupling-product-lem}
The systems of equations \eqref{interlevy} and \eqref{coupling-product-2} have a unique strong solution $(X_t,Y_t)_{t\geq 0}$ which is the coupling process associated to the coupling operator $\widetilde L$ above.
\end{lemma}

\begin{proof}
Recall that we assume the system of equations \eqref{interlevy} has a non-explosive and pathwise unique strong solution $(X_t)_{t\geq 0}$. As in Proposition \ref{2-prop-1}, we show that the sample paths of $(Y_t)_{t\geq 0}$ can be obtained by modifying those of the solution of the following equation:
  \begin{equation}\label{coupling-product-lem.1}
  d \tilde Y^i_t=b^i(\tilde Y_t)\,dt+ dZ^i_t,\quad \tilde Y^i_0=y^i,\quad 1\leq i\leq n.
  \end{equation}

Without loss of generality, we assume $n=2$. Denote by $Y^{(1)}_t$ the solution to \eqref{coupling-product-lem.1}.
For $i=1,2$, take independent random variables $\zeta_1^i$ and $\zeta_2^i$ which are uniformly distributed on $[0,1]$. Define the stopping times
  \begin{align*}
  \sigma_1^i=\inf\Big\{t>0: &\,
  \zeta_1^i\leq \frac12\Big[\rho_i\big(\big(Y^{(1),i}_t-X^i_t\big)_{\kappa^i}, \Delta Z^i_t\big)+ \rho_i\big(\big(X^i_t- Y^{(1),i}_t\big)_{\kappa^i}, \Delta Z^i_t\big)\Big]\Big\}
  \end{align*}
and
  \begin{align*}
  \sigma_2^i=\inf\!\Big\{t>\! \sigma_1^1\wedge \sigma_1^2: &\,
  \zeta_2^i\leq \frac12\Big[\rho_i\big(\big(Y^{(2),i}_t-X^i_t\big)_{\kappa^i}, \Delta Z^i_t\big)+ \rho_i\big(\big(X^i_t- Y^{(2),i}_t\big)_{\kappa^i}, \Delta Z^i_t\big)\Big]\Big\}  \end{align*}
for $i=1,2$. Then, using the equations \eqref{interlevy} and \eqref{coupling-product-1.5}, we can follow the proof of Proposition \ref{2-prop-1} with $\sigma_j$ replaced by $\sigma_j^1\wedge \sigma_j^2$ for $j=1,2$ respectively,  and also the argument of Proposition \ref{3-prop-2} to show the desired assertion.
\end{proof}

We can now present the

\begin{proof}[Proof of Proposition $\ref{P:prod}$]
By \eqref{coupling-product-1}, we have
  \begin{align*}
  \widetilde L d_{\psi,w}(|x-y|)&=\sum_{i=1}^n w_i\widetilde L_{i} \psi_i(|x^i-y^i|)+\sum_{i=1}^n w_i\psi'_i(|x^i-y^i|) \frac{\langle \gamma^i(x)-\gamma^i(y), x^i-y^i\rangle}{|x^i-y^i|}\\
  &\le \sum_{i=1}^n w_i\widetilde L_{i} \psi_i(|x^i-y^i|)+\sum_{i=1}^n w_i\psi'_i(0) |\gamma^i(x)-\gamma^i(y)|,
  \end{align*}
where in the inequality above we have used the fact that $\psi_i'(r)\le \psi_i'(0)$ for all $r\ge0$. Next, \eqref{con-remin} implies
  \begin{align*}
  \widetilde L d_{\psi,w}(|x-y|)&\le -\sum_{i=1}^n \lambda_i w_i\psi_i(|x^i-y^i|)+ \sum_{i=1}^n w_i \varepsilon_i\psi_i(|x^i-y^i|)\\
  &\le -\lambda\sum_{i=1}^n  w_i\psi_i(|x^i-y^i|)=-\lambda d_{\psi, w}(|x-y|).
  \end{align*}
The inequality above along with Theorem \ref{p-w} yields the first desired assertion. The second assertion just follows from the first one and the definition of $d_{\psi,w}$.
\end{proof}

In many applications, the perturbation $\gamma=(\gamma^1, \ldots, \gamma^n)$ satisfies an $l^1$-Lipschitz condition
$$\sum_{i=1}^n|\gamma^i(x)-\gamma^i(y)|\le \lambda\sum_{i=1}^n|x^i-y^i|,\quad x,y\in\R^d.$$ Using Propositions \ref{P:per} and \ref{P:prod},
we can easily get exponential contractivity in terms of
$W_{d_{\psi,1}}$ on product spaces and the corresponding
perturbation assertions of product models with respect to $W_{l^1}$.
These can be applied to the following system for interacting SDEs
with jumps
  $$dX_t^i=-\frac{1}{2}\nabla U(X_t^i)\,dt-\sum_{j=1}^na_{ij}\nabla V(X_t^i-X_t^j)\,dt+ dZ_t^i,\quad 1\le i\le n,$$
where $(Z_t^i)_{t\ge0}$ $(i=1,\ldots,n)$ are independent L\'evy processes in $\R^{k}$, $U\in C^2(\R^k)$ is strictly convex outside a given ball, the interaction potential $V$ is in $C^2(\R^k)$ with bounded second derivatives, and $a_{i,j}$, $1\le i,j\le n$, are finite real constants.

\section{Appendix: Properties of $\nu\wedge(\delta_x \ast \nu)$}

Recall that for any two finite measures $\mu_1$ and $\mu_2$ on $(\R^d,\Bb(\R^d))$,
  $$\mu_1\wedge\mu_2:=\mu_1-(\mu_1-\mu_2)^+,$$
where $(\mu_1-\mu_2)^{\pm}$ refers to the Jordan--Hahn decomposition of the signed measure $\mu_1-\mu_2$. In particular, $\mu_1\wedge\mu_2=\mu_2\wedge\mu_1$ and for any $A\in\mathscr{B}(\R^d)$,
  $$(\mu_1-\mu_2)^{+}(A)=\sup\{\mu_1(B)-\mu_2(B): B\subset A, B\in \mathscr{B}(\R^d)\}.$$
Thus, for any $A\in \mathscr{B}(\R^d)$,
  \begin{align*}
  (\mu_1\wedge\mu_2)(A) =&\ \mu_1(A)-\sup\{\mu_1(B)-\mu_2(B): B\subset A, B\in \mathscr{B}(\R^d)\}\\
  =&\inf\{\mu_1(A\setminus B)+\mu_2(B): B\subset A,  B\in \mathscr{B}(\R^d)\}.
  \end{align*}
The expression above can be extended to any measures (not necessarily finite) $\mu_1$ and $\mu_2$. From this, we can easily claim that

\begin{lemma}\label{L:mea}
Let $\mu_1$ and $\mu_2$ be two measures on $(\R^d,\Bb(\R^d))$. For any $x\in\R^d$, it holds that
  $$\delta_x\ast (\mu_1\wedge\mu_2)= (\delta_x\ast \mu_1)\wedge (\delta_x\ast \mu_2).$$
\end{lemma}

As a consequence of Lemma \ref{L:mea}, we have the following statement.

\begin{corollary}\label{C:mea}
Let $\nu$ be a L\'{e}vy measure on $(\R^d,\Bb(\R^d))$. Then, for any $x\in\R^d$,
  $$\delta_x\ast (\nu\wedge(\delta_{-x}\ast \nu))=\nu\wedge(\delta_x\ast \nu)$$
and so
  $$[\nu\wedge(\delta_{-x}\ast \nu)](\R^d)=[\nu\wedge(\delta_x\ast \nu)](\R^d).$$
\end{corollary}

\begin{proof}
 The first assertion follows from Lemma \ref{L:mea}. Moreover,
  \[ [\nu\wedge(\delta_{-x}\ast \nu)](\R^d)= [\delta_x\ast (\nu\wedge(\delta_{-x}\ast \nu))](\R^d)=[\nu\wedge(\delta_{x}\ast \nu)](\R^d). \qedhere \]
\end{proof}

We can also justify the following example.

\begin{example} \label{exnu}\it
Let $\nu(dz)=g(z)\,dz$ for some nonnegative measurable function $g$. Then for any $x\in\R^d$,
  $$\nu\wedge(\delta_{x}\ast \nu)(dz)=(g(z)\wedge g(z-x))\,dz.$$
\end{example}

Before proving the main result of this part, we present the following simple result.

\begin{lemma}\label{app-lem}
Assume that $\mu_1,\mu_2$ and $\nu$ are $\sigma$-finite measures on $(\R^d, \mathscr{B}(\R^d))$. If $\mu_1$ and $\mu_2$ are singular to each other, then
  $$(\mu_1 +\mu_2)\wedge \nu = \mu_1 \wedge \nu +\mu_2\wedge \nu.$$
\end{lemma}

\begin{proof}
There exists a Borel set $S\subset \R^d$ such that $\mu_1$ and
$\mu_2$ are supported on  $S$ and $S^c= \R^d\setminus S$,
respectively. Fix a set $A\in \mathscr{B}(\R^d)$. Then, for any
$B\in \mathscr{B}(\R^d)$ with $B\subset A$,
  \begin{align*}
  \mu_1(B) + \mu_2(B) + \nu(A\setminus B)
  =\ &\mu_1(B\cap S) + \mu_2(B\cap S^c) \\
  &+ \nu\big((A\cap S)\setminus (B\cap S)\big) + \nu\big((A\cap S^c)\setminus (B\cap S^c)\big).
  \end{align*}
Therefore,
  \begin{align*}
  \ [(\mu_1 +\mu_2)\wedge \nu](A) = &\inf\big\{\mu_1(B) + \mu_2(B) + \nu(A\setminus B): B\subset A,  B\in \mathscr{B}(\R^d)\big\} \\
  = & \inf\big\{\mu_1(B_1) + \mu_2(B_2) + \nu\big((A\cap S)\setminus B_1\big) + \nu\big((A\cap S^c)\setminus B_2\big): \\ &\hskip100pt B_1\subset A\cap S, B_2\subset A\cap S^c,  B_1, B_2\in \mathscr{B}(\R^d)\big\} \\
  = & \inf\big\{\mu_1(B_1) + \nu\big((A\cap S)\setminus B_1\big) : B_1\subset A\cap S, B_1\in \mathscr{B}(\R^d)\big\} \\
  &+ \inf\big\{\mu_2(B_2) + \nu\big((A\cap S^c)\setminus B_2\big) : B_2\subset A\cap S^c, B_2\in \mathscr{B}(\R^d)\big\} \\
  =& \ (\mu_1 \wedge \nu)(A\cap S) + (\mu_2 \wedge \nu)(A\cap S^c)\\
  =&\ (\mu_1 \wedge \nu)(A) + (\mu_2 \wedge \nu)(A),
  \end{align*}
where the last step is due to the facts that $\mu_1 \wedge \nu$ and
$\mu_2 \wedge \nu$ are supported respectively on $S$ and $S^c$.
\end{proof}

Finally we can prove

\begin{proposition}\label{ppp}
Assume that the L\'evy measure $\nu$ satisfies \eqref{th10} for some constant $\kappa_0>0$. Then, there is a nonnegative measurable function $\rho$ on $\R^d$ such that
  $$\nu(dz)\ge \rho(z)\,dz$$
and
  $$\inf_{x\in\R^d, |x|\le \kappa_0} \int_{\R^d}[\rho(z)\wedge \rho(z+x)]\,dz>0.$$
\end{proposition}

\begin{proof}
(1) We first prove that $\nu$ has a non-zero absolutely continuous part with respect to the Lebesgue measure. For any $x\in\R^d$, we have $\nu\ge \nu \wedge (\delta_x \ast \nu)$, so
  $$\nu\ge \int_{\overline{B(0,\kappa_0)}} [\nu \wedge (\delta_x \ast \nu)]\,U(dx)=:\tilde \nu,$$
where $U(dx)$ is the uniform distribution on the closed ball $\overline{B(0,\kappa_0)}$. Since, by \eqref{th10},
  $$\tilde \nu(\R^d)=\int_{\overline{B(0,\kappa_0)}}[\nu \wedge (\delta_x \ast \nu)](\R^d)\,U(dx)\ge \inf_{|x|\le \kappa_0}[\nu \wedge (\delta_x \ast \nu)](\R^d)>0,$$
$\tilde \nu$ is a non-zero measure on $(\R^d,\mathscr{B}(\R^d))$. For any $A\in \mathscr{B}(\R^d)$, we denote by ${\rm Leb}(A)$ the Lebesgue measure of $A$. If ${\rm Leb}(A)=0$, then
  $$\tilde \nu(A)=\int_{\overline{B(0,\kappa_0)}}[\nu \wedge (\delta_x \ast\nu)](A)\,U(dx)\le \int_{\overline{B(0,\kappa_0)}}(\delta_x \ast\nu)(A)\,U(dx)=\nu\ast U(A)=0.$$
Combining both conclusions above, we prove the desired assertion.

(2) We deduce from step (1) that there is a nonnegative measurable function $\rho$ on $\R^d$ such that
  $$\nu(dz)= \rho(z)\,dz+\nu^{\rm s}(dz),$$
where $\nu^{\rm s}$ is the singular part with respect to the Lebesgue measure. Suppose that
  $${\rm Leb}\big(\{x\in \overline{B(0,\kappa_0)}:[\nu^{\rm s} \wedge (\delta_x \ast \nu^{\rm s})](\R^d)>0\} \big)>0.$$
Following the same argument above with the measure $U$ replaced by the uniform distribution on the set $\{x\in \overline{B(0,\kappa_0)}:[\nu^{\rm s} \wedge (\delta_x \ast \nu^{\rm s})](\R^d)>0 \}$, we can claim that $\nu^{\rm s}$ has a non-zero absolutely continuous part with respect to the Lebesgue measure, which is a contradiction. Therefore,
  \begin{equation}\label{ppps1}
  \essinf\limits_{{x\in\R^d, |x|\le \kappa_0}}[\nu^{\rm s} \wedge (\delta_x \ast\nu^{\rm s})](\R^d)=0.
  \end{equation}

Next, we denote by $\nu^{\rm ac}(dz) = \rho(z)\,dz$ the absolutely continuous part of $\nu$. It is clear that $(\delta_x \ast \nu)^{\rm ac} = \delta_x \ast \nu^{\rm ac}$ and $(\delta_x \ast \nu)^{\rm s} = \delta_x \ast \nu^{\rm s}.$ These properties along with the shift invariance of the Lebesgue measure yield that $\nu^{\rm s}\wedge (\delta_x \ast \nu^{\rm ac})=0$ and $\nu^{\rm ac}\wedge (\delta_x \ast \nu^{\rm s})=0$. By Lemma \ref{app-lem},
  \begin{equation}\label{ppps2}
  \nu\wedge(\delta_x\ast \nu) =\nu^{\rm ac} \wedge (\delta_x\ast \nu^{\rm ac})+\nu^{{\rm s}} \wedge (\delta_x\ast \nu^{{\rm s}}).
  \end{equation}
Therefore, we deduce from \eqref{th10}, \eqref{ppps1}, \eqref{ppps2} and Example \ref{exnu}
that
  \begin{align*}
  \essinf\limits_{{x\in\R^d, |x|\le \kappa_0}} [\nu\wedge (\delta_x\ast \nu)](\R^d)=&\essinf\limits_{{x\in\R^d, |x|\le \kappa_0}} [\nu^{\rm ac} \wedge (\delta_x\ast \nu^{\rm ac})](\R^d) \\
  = &\essinf\limits_{{x\in\R^d, |x|\le \kappa_0}} \int_{\R^d}[\rho(z)\wedge \rho(z+x)]\,dz>0.
  \end{align*}

To conclude the last assertion, it suffices to prove that
  \begin{equation}\label{ppp-0}
  \lim_{x\to x_0} \int_{\R^d} \rho(z)\wedge \rho(x+z)\, dz= \int_{\R^d} \rho(z)\wedge \rho(x_0+z)\, dz\quad \mbox{for any } 0<|x_0|\leq \kappa_0.
  \end{equation}
Fix any $\delta\in (0,\kappa_0)$ and $x_0\in \R^d$ such that $\delta\leq |x_0|\leq \kappa_0$. For any $x\in \R^d$ with $|x-x_0|\leq \delta/4$, we have
  $$\aligned
  &\hskip4pt \bigg|\int_{\R^d} \rho(z)\wedge \rho(x+z)\, dz- \int_{\R^d} \rho(z)\wedge \rho(x_0+z)\, dz\bigg|\\
  &\leq \bigg|\int_{\{|z|\leq \frac\delta 2\}}\big[ \rho(z)\wedge \rho(x+z)- \rho(z)\wedge \rho(x_0+z)\big]\, dz\bigg|\\
  &\hskip14pt + \bigg|\int_{\{|z|> \frac\delta 2\}}\big[ \rho(z)\wedge \rho(x+z)- \rho(z)\wedge \rho(x_0+z)\big]\, dz\bigg|\\
  &=: I_1+I_2.\endaligned$$

We first estimate $I_1$. Using the simple formulae $a\wedge b=\frac12(a+b-|a-b|)$ and $\big||a|-|b|\big|\leq |a-b|$, we obtain
  $$|\rho(z)\wedge \rho(x+z)- \rho(z)\wedge \rho(x_0+z)|\leq |\rho(x+z)-\rho(x_0+z)|.$$
Hence
  $$\aligned
  I_1& \leq \int_{\{|z|\leq \frac\delta 2\}} |\rho(x+z)-\rho(x_0+z)| \, dz =\int_{\{|y-x_0|\leq \frac\delta 2\}} |\rho(x-x_0+y)-\rho(y)| \, dy\\
  &\leq \int_{\{|y|\geq \frac\delta 2 \}} |\rho(x-x_0+y)-\rho(y)| \, dy.
  \endaligned$$
Since the function $\R^d\ni y\mapsto \I_{\{|y|\ge \frac\delta 4 \}} \rho(y)$ is integrable, for any $\varepsilon>0$, there exists a function $\varphi:=\varphi_\varepsilon\in C_c(\R^d)$ such that
  $$\int_{\R^d}\big|\I_{\{|y|\ge \frac\delta 4 \}} \rho(y) -\varphi(y)\big|\,d y<\varepsilon.$$
Note that
  $$\aligned
 \int_{\{|y|\geq \frac\delta 2 \}} |\rho(x-x_0+y)-\varphi(x-x_0+y)| \, dy&=\int_{\{|z-(x-x_0)|\geq \frac\delta 2 \}} |\rho(z)-\varphi(z)| \, dz \\
 &\leq \int_{\{|z|\geq \frac\delta 4 \}} |\rho(z)-\varphi(z)| \, dz,
  \endaligned$$
where the last inequality is due to $|x-x_0|\leq \delta/4$. Thus,
  $$\aligned
  I_1&\leq \int_{\{|y|\geq \frac\delta 2 \}} |\rho(x-x_0+y)-\varphi(x-x_0+y)| \, dy\\
  &\quad +\int_{\{|y|\geq \frac\delta 2 \}} |\varphi(x-x_0+y)-\varphi(y)| \, dy +\int_{\{|y|\geq \frac\delta 2 \}} |\varphi(y)-\rho(y)| \, dy\\
  &\leq 2\int_{\R^d}\big|\I_{\{|y|\ge \frac\delta 4 \}} \rho(y) -\varphi(y)\big|\,d y +\int_{\{|y|\geq \frac\delta 2 \}} |\varphi(x-x_0+y)-\varphi(y)| \, dy\\
  &< 2\varepsilon +  \int_{\R^d} |\varphi(x-x_0+y)-\varphi(y)| \, dy.
  \endaligned$$
Therefore,
  \begin{equation}\label{ppp-1}
  \limsup_{x\to x_0} I_1 \leq 2\varepsilon.
  \end{equation}

Next,
  $$I_2=\bigg|\int_{\{|y-x|> \frac\delta 2\}} \rho(y-x)\wedge \rho(y)\, dy - \int_{\{|y-x_0|> \frac\delta 2\}} \rho(y-x_0)\wedge \rho(y)\, dy\bigg|.$$
To simplify the notations, denote by $B_1=\{|y-x|> \frac\delta 2\}$ and $B_2=\{|y-x_0|> \frac\delta 2\}$. Then
  $$\aligned
  I_2&\leq \int_{B_1\setminus B_2} \rho(y-x)\wedge \rho(y)\, dy + \int_{B_2\setminus B_1} \rho(y-x_0)\wedge \rho(y)\, dy\\
  &\hskip14pt + \bigg|\int_{B_1\cap B_2}\big[ \rho(y-x)\wedge \rho(y)- \rho(y-x_0)\wedge \rho(y)\big]\, dy\bigg|\\
  &=:\! I_{2,1}+I_{2,2}+I_{2,3}.
  \endaligned$$
It is easy to see that
  $$I_{2,1} \leq \int_{B_1\setminus B_2} \rho(y-x)\, dy= \int_{\{|z|>\frac\delta 2 \}\cap \{|x-x_0+z|\leq \frac\delta 2 \}} \rho(z)\, dz\to 0$$
as $x\to x_0$, since $\rho(z) \I_{\{|z|>\frac\delta 2 \}}$ is integrable. Similarly,  $I_{2,2}\to 0$ as $x\to x_0$. Finally, analogous to the treatment of $I_1$, we have
  $$\aligned
  I_{2,3}&\leq \int_{B_1\cap B_2} |\rho(y-x) - \rho(y-x_0)|\, dy =\int_{\{|z-x+x_0|>\frac\delta 2\}\cap \{|z|>\frac\delta 2 \}} |\rho(z-x+x_0)-\rho(z)|\, dz\\
  &\leq \int_{\{|z|>\frac\delta 2 \}} |\rho(z-x+x_0)-\rho(z)|\, dz.
  \endaligned$$
In the same way as for \eqref{ppp-1}, we have $\limsup_{x\to x_0} I_{2,3}\leq 2 \varepsilon$, which, combined with the assertions for $I_{2,1}$ and $I_{2,2}$, implies
  $$\limsup_{x\to x_0} I_{2}\leq 2 \varepsilon.$$
This together with \eqref{ppp-1} again leads to \eqref{ppp-0}, since $\varepsilon>0$ is arbitrary.
\end{proof}

\medskip

\noindent \textbf{Acknowledgements.} The authors would like to thank
Professors Andreas Eberle and Feng-Yu Wang for their valuable comments  on an early version of this paper.
They are also very grateful to the two referees for careful reading of the manuscript and numerous suggestions of corrections;
many of them are adopted by us in the discussions of main results in Section \ref{section1} and in the comments on the construction of our coupling in Section \ref{sectionse2}.
The research of Dejun Luo is supported by the National Natural Science
Foundation of China (Nos. 11571347, 11688101), the Seven Main
Directions (Y129161ZZ1) and the Special Talent Program of the
Academy of Mathematics and Systems Science, Chinese Academy of
Sciences.   The research of Jian Wang is supported by National
Natural Science Foundation of China (No.\ 11522106), Fok Ying Tung
Education Foundation (No.\ 151002), and the Program for Probability and Statistics: Theory and Application (No. IRTL1704).

\end{document}